\documentclass [11pt] {amsart}
\usepackage{bbm}
\usepackage{amssymb}
\usepackage{epic}
\usepackage{ecltree}
\usepackage{tikz}
\usepackage{color}

\newtheorem{theorem}{Theorem}
\newtheorem{lemma}[theorem]{Lemma}
\newtheorem{problem}[theorem]{Problem}
\newtheorem{proposition}[theorem]{Proposition}
\newtheorem{corollary}[theorem]{Corollary}

\newtheorem*{theoremA}{Theorem A}
 \newtheorem*{theoremB}{Theorem B}
\theoremstyle{definition}
\newtheorem{definition}[theorem]{Definition}

\newtheorem{remark}[theorem]{Remark}

\newcommand{\supp}{{\rm supp}\hskip0.02cm}
\newcommand{\1}{\mathbf{1}}
\newcommand{\tp}{{\rm TP}\hskip0.02cm}
\newcommand{\tc}{{\rm TC}\hskip0.02cm}
\def\lip{\hskip0.02cm{\rm Lip}\hskip0.01cm}

\begin{document}

\title[Analysis on Laakso graphs]{Analysis on Laakso graphs with application to the structure of  transportation cost spaces}

 \author{S.\ J.\ Dilworth, Denka  Kutzarova and Mikhail I. Ostrovskii}

\address{Department of Mathematics, University of South Carolina, Columbia, SC 29208, U.S.A.}
\email{dilworth@math.sc.edu}

\address{Institute of Mathematics and Informatics, Bulgarian Academy of
Sciences, Bulgaria} \curraddr{Department of Mathematics, University of Illinois
at Urbana-Champaign, Urbana, IL 61801, U.S.A.} \email{denka@math.uiuc.edu}

\address{Department of Mathematics and Computer Science, St. John's
University, 8000 Utopia Parkway, Queens, NY 11439, USA}
  \email{ostrovsm@stjohns.edu}

\begin{abstract}  This article is a continuation of our article in [Canad. J. Math. Vol. 72 (3), (2020), pp. 774--804].
We construct orthogonal bases of the cycle and cut spaces of the
Laakso graph $\mathcal{L}_n$. They are used  to  analyze
projections from the edge space onto the cycle space and to obtain
reasonably sharp estimates of the projection constant of
$\operatorname{Lip}_0(\mathcal{L}_n)$, the space of Lipschitz
functions on $\mathcal{L}_n$. We deduce that the Banach-Mazur
distance from $\tc(\mathcal{L}_n)$, the transportation cost space
of $\mathcal{L}_n$,  to   $\ell_1^N$ of the same dimension is at
least  $(3n-5)/8$, which is the analogue of a result from [op.
cit.] for the diamond graph $D_n$. We calculate the exact
projection constants of $\operatorname{Lip}_0(D_{n,k})$, where
$D_{n,k}$ is the diamond graph of branching $k$. We also provide
simple examples of finite metric spaces, transportation cost
spaces on which contain $\ell_\infty^3$ and $\ell_\infty^4$
isometrically.
\end{abstract}

\maketitle

\smallskip


\section{Introduction}

\subsection{Definitions and background}

Let $(X,d)$ be a metric space. Consider a real-valued finitely
supported function $f$ on $X$ with a zero sum, that is,
$\sum_{v\in \supp f}f(v)=0$. A natural and important
interpretation of such a function, is considering it as a
\emph{transportation problem}: one needs to transport certain
product from locations where $f(v)>0$ to locations where $f(v)<0$.

One can easily see that $f$ can be  represented as
\begin{equation}\label{E:TranspPlan} f=a_1(\1_{x_1}-\1_{y_1})+a_2(\1_{x_2}-\1_{y_2})+\dots+
a_n(\1_{x_n}-\1_{y_n}),\end{equation} where $a_i\ge 0$,
$x_i,y_i\in X$, and $\1_u(x)$ for $u\in X$ is the {\it indicator
function} of $u$, defined by
\[\1_u(x)=\begin{cases} 1 &\hbox{ if }x=u,\\ 0 &\hbox{ if }x\ne u.
\end{cases} \]

We call each such representation a \emph{transportation plan} for
$f$, and it can be interpreted as a plan of moving $a_i$ units of the
product from $x_i$ to $y_i$. The \emph{cost} of the transportation
plan \eqref{E:TranspPlan} is defined as $\sum_{i=1}^n
a_id(x_i,y_i)$.

\begin{remark} It is worth mentioning that in our discussion
transportation plans are allowed to be {\it fake plans}, in the
sense that it can happen that there is no product in $x_i$ in
order to make the delivery to $y_i$. To see what we mean consider
a metric space containing three distinct points $x,y,z$. Then
$(\1_x-\1_y)+(\1_y-\1_z)+(\1_z-\1_x)$ is a transportation plan for
function $0$ (null transportation problem, nothing is needed or
available), although there is no product in $x$ to be delivered to
$y$. However, it is easy to show that the defined below {\it
optimal transportation plans} can be implemented.
\end{remark}

We denote the real vector space of all transportation problems by
$\tp(X)$. We introduce the \emph{transportation cost norm} (or
just \emph{transportation cost}) $\|f\|_{\tc}$ of a transportation
problem $f$ as the infimum of costs of transportation plans
satisfying \eqref{E:TranspPlan}. Using the triangle inequality and
compactness it is easy to show that the infimum of costs of
transportation plans for $f$ is attained. A transportation plan
for $f$ whose cost is equal to $\|f\|_\tc$ is called an
\emph{optimal transportation plan}.  The completion of the normed
space $(\tp(X),\|\cdot\|_\tc)$ is called a \emph{transportation
cost space} and is denoted by $\tc(X)$.

We use the standard terminology of Banach space theory
\cite{BL00}, graph theory \cite{Die17}, and the theory of metric
embeddings \cite{O}.

Transportation cost spaces are of interest in many areas and are
studied under many different names (we list some of them in the
alphabetical order: Arens-Eells space, earth mover distance,
Kantorovich-Rubinstein distance, Lipschitz-free space, Wasserstein
distance). We prefer to use the term {\it transportation cost
space} since it makes the subject of this work instantly clear to
a wide circle of readers and it also reflects the historical
approach leading to these notions (see \cite{Kan42,KG49}).
Interested readers can find a review of the main definitions,
notions, facts, terminology and historical notes pertinent to the
subject in \cite[Section 1.6]{OO19}.

By a \emph{pointed metric space} we mean a metric space $(X,d_X)$
with a \emph{base point}, denoted by $O$. For a pointed metric
space $X$ with a base point at $O$ by $\lip_0(X)$ we denote the
space of all Lipschitz functions $f:X\to\mathbb{R}$ satisfying
$f(O)=0$. It is not difficult to check that $\lip_0(X)$ is a
Banach space with respect to the norm $\|f\|=\lip(f)$ ($\lip(f)$
is the Lipschitz constant of $f$). As is well known
$\tc(X)^*=\lip_0(X)$ (see e.g. \cite[Section
10.2]{O}).

One of the main goals of this paper is to study the geometry of
the spaces $\tc(X)$. We are interested mostly in the case where
$X$ is finite. We would like to mention that for finite $X$, the
space $\tc(X)$ is an $\ell_1$-like space in the sense that is has
three qualities which make it close to $\ell_1^{|X|-1}$.

(1) It has a $1$-complemented subspace isometric to
$\ell_1^{\lceil |X|/2\rceil}$, see \cite{KMO}   (a
weaker version was proved earlier in \cite{DKO}).

(2) It admits a linear embedding into $L_1[0,1]$ with
distortion $\le C\ln |X|$, see \cite{Cha02,FRT04,IT03}.
 Although this result is known since 2003, it seems
that the only source where one can find its published proof is
\cite[Theorem 15]{BMSZ20+}.

(3) It is a quotient of $\ell_1^d$ with $d\le |X|^2$, see
\cite{OO20+}. Another proof and a more precise statement can be
found in Section \ref{S:Tree}.
\medskip

However, $\tc(X)$ is isometric to $\ell_1^{|X|-1}$ if and only if
$X$ is a weighted tree. This result can be derived from the
general result of \cite{DKP16}. Apparently the finite case of this
result can be considered as folklore, for convenience of the
readers we give a direct proof of the ``only if'' part (for finite
case) in Section \ref{S:Tree}, the ``if'' part can be found in
\cite[Proposition 2.1]{DKO}.

One of the important problems about transportation cost spaces is
the following \cite[Problem 2.6]{DKO}:

\begin{problem}\label{P:CloseToell1} It would be very interesting to find a condition on a finite metric space $M$ which is equivalent
to the condition that the space $\tc(M)$ is Banach-Mazur close to
$\ell_1^n$ of the corresponding dimension. It is not clear whether
it is feasible to find such a condition.
\end{problem}

In \cite{DKO} we investigated this problem for large recursive
families of graphs which include well-known families of diamond
and Laakso graphs.

The main goal of this paper is further development of analysis in
the space of functions on diamond and Laakso graphs in order to
sharpen results of \cite{DKO}. Let us remind the definitions of
these families of graphs.

\begin{definition}[Diamond graphs]\label{D:Diamonds}
Diamond graphs $\{D_n\}_{n=0}^\infty$ are defined recursi\-ve\-ly:
The {\it diamond graph} of level $0$ has two vertices joined by an
edge of length $1$ and is denoted by $D_0$. The {\it diamond
graph} $D_n$ is obtained from $D_{n-1}$ in the following way.
Given an edge $uv\in E(D_{n-1})$, it is replaced by a
quadrilateral $u, a, v, b$, with edges $ua$, $av$, $vb$, $bu$.
(See Figure \ref{F:Diamond2}.)
\end{definition}

Apparently Definition \ref{D:Diamonds} was first introduced in
\cite{GNRS04}.

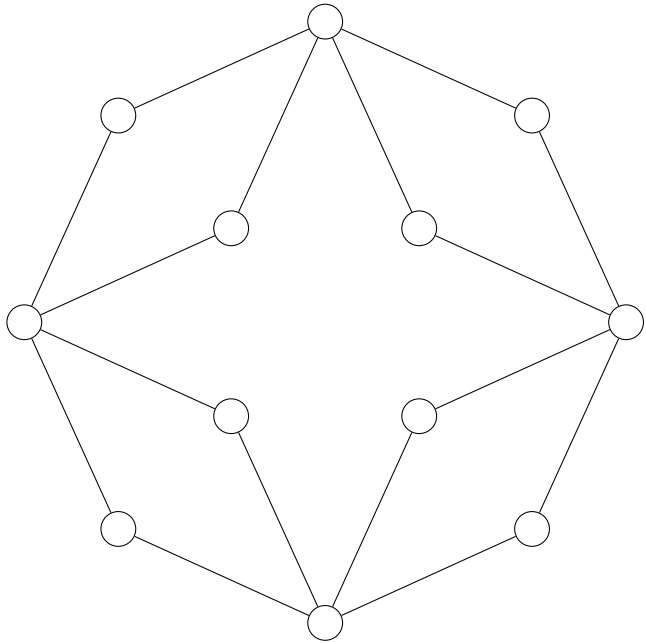
\begin{figure}
\begin{center}
{
\begin{tikzpicture}
  [scale=.25,auto=left,every node/.style={circle,draw}]
  \node (n1) at (16,0) {\hbox{~~~}};
  \node (n2) at (5,5)  {\hbox{~~~}};
  \node (n3) at (11,11)  {\hbox{~~~}};
  \node (n4) at (0,16) {\hbox{~~~}};
  \node (n5) at (5,27)  {\hbox{~~~}};
  \node (n6) at (11,21)  {\hbox{~~~}};
  \node (n7) at (16,32) {\hbox{~~~}};
  \node (n8) at (21,21)  {\hbox{~~~}};
  \node (n9) at (27,27)  {\hbox{~~~}};
  \node (n10) at (32,16) {\hbox{~~~}};
  \node (n11) at (21,11)  {\hbox{~~~}};
  \node (n12) at (27,5)  {\hbox{~~~}};

  \foreach \from/\to in {n1/n2,n1/n3,n2/n4,n3/n4,n4/n5,n4/n6,n6/n7,n5/n7,n7/n8,n7/n9,n8/n10,n9/n10,n10/n11,n10/n12,n11/n1,n12/n1}
    \draw (\from) -- (\to);

\end{tikzpicture}
} \caption{Diamond $D_2$.}\label{F:Diamond2}
\end{center}
\end{figure}

Let us count some parameters associated with  the graphs $D_n$. Denote
by $V(D_n)$ and $E(D_n)$ the vertex set and edge set of $D_n$,
respectively.  Note that:

\begin{enumerate}

\item[{\bf (a)}] $|E(D_n)|=4^n$.

\item[{\bf (b)}] $|V(D_{n+1})|=|V(D_n)|+2|E(D_n)|$.\label{P:(b)}

\end{enumerate}

Hence $|V(D_n)|=2(1+\sum_{i=0}^{n-1}4^i)$.
\medskip

\begin{definition}[Multibranching diamonds]\label{D:BranchDiam}
For  any integer  $k\ge 2$, we define $D_{0,k}$ to be the graph
consisting of two vertices joined by one edge. For any
$n\in\mathbb{N}$, if the graph $D_{n-1,k}$  is already defined,
the graph $D_{n,k}$ is defined as the graph obtained from
$D_{n-1,k}$ by replacing each edge $uv$ in $D_{n-1,k}$ by a set of
$k$ independent paths of length $2$ joining $u$ and $v$. We endow
$D_{n,k}$ with the shortest path distance. We call
$\{D_{n,k}\}_{n=0}^\infty$ {\it diamond graphs of branching $k$},
or {\it diamonds of branching} $k$.
\end{definition}

Definition \ref{D:BranchDiam} was introduced in \cite{LR}. Note that:

\begin{enumerate}

\item[{\bf (a)}] $|E(D_{n,k})|= (2k)^n$.

\item[{\bf (b)}]
$|V(D_{n+1,k})|=|V(D_{n,k})|+k|E(D_{n,k})|$.\label{P:(b)brk}

\end{enumerate}

Hence $|V(D_{n,k})|=2+k\sum_{i=0}^{n-1}(2k)^i$.
\medskip

\begin{figure}
\begin{center}
{
\begin{tikzpicture}
  [scale=.25,auto=left,every node/.style={circle,draw}]
  \node (n1) at (16,0) {\hbox{~~~}};
  \node (n2) at (16,8)  {\hbox{~~~}};
  \node (n3) at (12,16) {\hbox{~~~}};
  \node (n4) at (20,16)  {\hbox{~~~}};
  \node (n5) at (16,24)  {\hbox{~~~}};
  \node (n6) at (16,32) {\hbox{~~~}};

\foreach \from/\to in {n1/n2,n2/n3,n2/n4,n4/n5,n3/n5,n5/n6}
    \draw (\from) -- (\to);

\end{tikzpicture}
} \caption{Laakso graph $\mathcal{L}_1$.}\label{F:Laakso}
\end{center}
\end{figure}
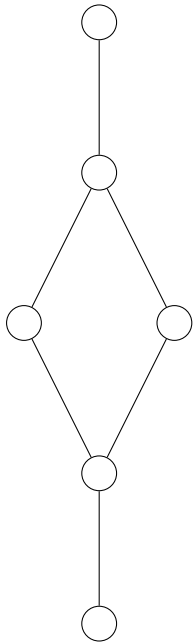

\begin{definition}\label{D:Laakso} Laakso graphs
$\{\mathcal{L}_n\}_{n=0}^\infty$ are defined recursively: The {\it
Laakso graph} of level $0$ has two vertices joined by an edge of
length $1$ and is denoted $\mathcal{L}_0$. The {\it Laakso graph}
$\mathcal{L}_n$ is obtained from $\mathcal{L}_{n-1}$ according to
the following procedure. Each edge $uv\in E(\mathcal{L}_{n-1})$ is
replaced by the graph $\mathcal{L}_1$ exhibited in Figure
\ref{F:Laakso}, the vertices $u$ and $v$ are identified with the
vertices of degree $1$ of $\mathcal{L}_1$.
\end{definition}

Definition \ref{D:Laakso} was introduced in \cite{LP}, where an
idea of Laakso \cite{Laa00} was used. Note that:

\begin{enumerate}

\item[{\bf (a)}] $|E(\mathcal{L}_n)|=6^n$.

\item[{\bf (b)}]
$|V(\mathcal{L}_{n+1})|=|V(\mathcal{L}_n)|+4|E(\mathcal{L}_n)|$.\label{P:(b)Laakso}

\end{enumerate}

Hence $|V(\mathcal{L}_n)|= 2+  4\sum_{i=0}^{n-1}6^i$.
\medskip

Diamond and Laakso  graphs play important roles in Metric Geometry
as examples/counterexamples to many natural questions. One of the
reasons for interest in the families of graphs introduced in
Definitions \ref{D:Diamonds}-\ref{D:Laakso} is that their
bilipschitz embeddability characterizes non-superreflexive Banach
spaces \cite{JS,Ost11,OR17}. In \cite{OO} it was shown that Laakso
graphs are incomparable with diamond graphs in the following
sense: elements of none of these families admit bilipschitz
embeddings into the other family with uniformly bounded
distortions.

We need the following description of $\tc(X)$ in the case where
$X$ is a vertex set of an unweighted graph with its graph
distance. Let $G=(V(G),E(G))=(V,E)$ be a finite graph. Let
$\ell_1(E)$, $\ell_2(E)$, and $\ell_\infty(E)$  be the spaces of
real-valued functions on $E$ with the norms $\|f\|_1=\sum_{e\in
E}|f(e)|$, $\|f\|_2=\left(\sum_{e\in E}|f(e)|^2\right)^{\frac12}$,
and $\|f\|_\infty=\max_{e\in E}|f(e)|$, respectively. We also
consider the inner product $\langle f,g\rangle$ associated with
$\|f\|_2$.

We consider an arbitrary chosen orientation on $E$, so each edge
of $E$ is a directed edge. We denote by $e^+$ and $e^-$ the {\it
head} and {\it tail} of an oriented edge $e$, respectively. The
choice of orientation affects some of the objects which we
introduce, but does not affect the final results. Such orientation
is usually called {\it reference orientation}.

For a directed cycle $C$ in $E$ (we mean that the cycle can be
``walked around'' following the direction, which is not related
with the orientation of $E$) we introduce the {\it signed
indicator function} of $C$ by
\begin{equation}\label{E:SignInd}
\chi_C(e)=\begin{cases} 1 & \hbox{ if }e\in C\hbox{ and its
orientations in $C$ and $G$ are the same}\\
-1 & \hbox{ if }e\in C\hbox{ but its orientations in $C$ and $G$
are different}
\\
0 & \hbox{ if }e\notin C.
\end{cases}\end{equation}

The {\it cycle space} $Z(G)$ of $G$ is the subspace of $\ell_1(E)$
spanned by the signed indicator functions of all cycles in $G$.
The orthogonal complement of $Z(G)$ in $\ell_2(E)$ is called the
{\it cut space}.

We will use the fact (\cite[Proposition 10.10]{O}) that $\tc(G)$
for unweighted graphs $G$ is isometrically isomorphic to the
quotient of $\ell_1(E)$ over $Z(G)$:
\begin{equation}\label{E:LFunweigh}\tc(G)=\ell_1(E)/Z(G)\end{equation}

The paper \cite{OO20+} contains a generalization of
\eqref{E:LFunweigh} for weighted graphs, and thus for arbitrary
finite metric spaces.

For convenience of the readers we give a simple proof of
\eqref{E:LFunweigh}.

\begin{proof} Observe that if $G=(V,E)$ is endowed with  a  reference orientation, each function
$f\in\ell_1(E)$ can be regarded as transportation plan given by
\[\sum_{e\in E}f(e)(\1_{e^-}-\1_{e^+}),
\]
and the cost of this plan is $\|f\|_1$ (note that $f(e)$ can be
negative, so this transportation  plan is not necessarily in the
form \eqref{E:TranspPlan}).

In turn, each such transportation plan gives (after summation) the
transportation problem which it solves. Thus (for any fixed
reference orientation) there is a natural linear map
$T:\ell_1(E)\to\tp(G)=\tc(G)$ (we consider finite graphs). The
statement in the previous paragraph implies that $\|Tf\|_\tc\le
\|f\|_1$.

It remains to show that for each transportation problem
$x\in\tc(G)$ there is $f\in\ell_1(E)$, such that $Tf=x$ and
$\|f\|_1=\|x\|_\tc$.

Let $\sum_{i=1}^na_i(\1_{x_i}-\1_{y_i})$ be an optimal
transportation plan for $x$. Since pairs $x_iy_i$ do not
necessarily form edges, this optimal transportation plan does not
immediately and naturally correspond to a vector in $\ell_1(E)$.
Nevertheless, by the definition of a graph distance, for each such
pair $x_iy_i$, we can find a shortest path
$u_{0,i},u_{1,i},\dots,u_{m(i),i}$ in $G$ with $u_{0,i}=x_i$,
$u_{m(i),i}=y_i$, each pair $u_{j-1,i}u_{j,i}$ $(j=1,\dots,m(i))$
being an edge in $G$, and $m(i)=d(x_i,y_i)$.

Then, as is easy to see,
\[\sum_{i=1}^n\sum_{j=1}^{m(i)}a_i(\1_{u_{j-1,i}}-\1_{u_{j,i}}),\]
is also an optimal transportation plan for $x$ and this plan
corresponds to a vector $f$ in $\ell_1(E)$ with
$\|f\|_1=\|x\|_\tc$.

The correspondence is the following: $f(e)=0$ is $e$ is not of the
form $u_{j-1,i}u_{j,i}$ for some $i$ and $j$, and $f(e)
=\theta(e,i,j)a_i$, if $e$ is of the form $u_{j-1,i}u_{j,i}$,
where $\theta(e,i,j)=1$ if $u_{j-1,i}$ is the tail of $e$ and
$\theta(e,i,j)=-1$ if $u_{j-1,i}$ is the head of $e$.
\end{proof}

\subsection{Results from \cite{DKO} on iteratively defined graphs}

Let us recall two  results from \cite{DKO} which are relevant to the present work.

A directed graph $B$ having two distinguished vertices which we
call {\it top} and {\it bottom}, generates a recursive family
$\{B_n\}_{n=0}^\infty$ as follows:
\begin{itemize} \item The graph $B_0$ consists of one directed
edge. \item For $n\ge1$, $B_n$ is obtained from $B_{n-1}$ by
replacing each edge by a copy of $B$, identifying bottom of $B$
with the tail of the edge and top of $B$ with the head of the
edge. Edges of $B_{n}$ inherit their directions from the
corresponding copies of $B$.
\end{itemize}

In \cite{DKO} we considered the recursive families corresponding to directed graphs $B$ satisfying certain natural conditions listed in \cite[Section~4.1]{DKO}), which  include the multibranching diamond and Laakso graphs defined above.

\begin{theoremA}\label{T:B_n} \cite[Theorem 4.2]{DKO} If the directed graph $B$ satisfies  the conditions of \cite[Section~4.1]{DKO}   and $\{B_n\}_{n=0}^\infty$ is
the corresponding recursively defined family then the Banach-Mazur
distance to $\ell_1^{d(n)}$ satisfies
\[d_{BM}(\tc(B_n), \ell_1^{d(n)})\ge\frac{cn}{\ln n}\] for $n\ge 2$ and some absolute constant $c>0$, where $d(n)$ is the dimension of
$\tc(B_n)$.
\end{theoremA}
The $\ln n$ factor in Theorem~A was removed for the case of
multibranching diamond graphs and an upper bound was also proved.
\begin{theoremB}\label{T:LFmultL1} \cite[Theorem 6.10]{DKO} The Banach-Mazur distance $d_{n,k}$ from the transportation cost  space $\tc(D_{n,k})$ to the $\ell_1^N$ space of the same dimension satisfies $$4n+4 \ge d_{n,k} \ge \frac{k-1}{2k}n.$$ \end{theoremB}
\subsection{Statement of results}
Our  main goal is to investigate the analogue of Theorem~B for the
Laakso graph $\mathcal{L}_n$.   In  Section~\ref{sec:
applications}   we prove the lower bound of $(3n-5)/8$ for the
Banach-Mazur distance from $\tc(\mathcal{L}_n)$ to $\ell_1^N$
(Corollary~\ref{cor: BMlowerbound}). This removes the $\ln n$
factor of Theorem~A and is the analogue of the lower bound  in
Theorem~B. However, we have not succeeded in proving a comparable
 (e.g. $O(n^a)$) upper bound. The obstacle to proving an analogue of the
upper bound in Theorem~B is explained in Section~7.

Our analysis of  $\tc(\mathcal{L}_n)$ is based on the fact (see
\eqref{E:LFunweigh}) that $\tc(G)$ is isometrically isomorphic to
$E(G)/Z(G)$. In Section~\ref{E:LFunweigh} we construct orthogonal
basis vectors for the cycle and cut spaces and in
Section~\ref{sec: normcalcs} we compute their norms. They are used
in Section~\ref{sec: projectionPn} to construct a projection $P_n$
from the edge space  onto the cycle space of relatively small norm
(Theorem~\ref{thm: goodprojection}). In Section~\ref{sec:
invariant} we show that $P_n$ is close to being of minimal norm
(Theorem~\ref{thm: projnormlowerbound}).  To prove this, we use
the method of invariant projections as in Gr\"unbaum \cite{Gru60},
Rudin \cite{Rud62} and Andrew \cite{A},  and  analyze projections
that are invariant with respect to a certain group of isometries
of the edge space.

Let $X$ be a finite-dimensional normed space and let $X_1$ be any subspace of $\ell_\infty$  that is isometrically isomorphic to $X$. Recall that the \textit{projection constant} of $X$, denoted $\lambda(X)$, is defined by
$$ \lambda(X) = \inf \{\|P\| \colon P \colon \ell_\infty \rightarrow \ell_\infty\text{ is a projection with range $X_1$}\}.$$
(Note that $\lambda(X)$ is independent of the choice of $X_1$.)

In  Section~\ref{sec: applications} we deduce from Theorems~\ref{thm: goodprojection} and \ref{thm: projnormlowerbound} reasonably sharp estimates of the projection constant of the space of Lipschitz functions on $\mathcal{L}_n$ (Theorem~\ref{thm: projconstantLaakso}). We also present the  results
described above on the transportation cost space of $\mathcal{L}_n$.  In Section~\ref{sec: diamondgraphs} we sharpen the proof of Theorem~B from \cite{DKO} to obtain the exact projection constant of the space of Lipschitz functions on $D_{n,k}$.

In Section \ref{S:Tree}, for the convenience of the reader   we
give a direct proof  in the finite case  that if $\tc(X)$ is
isometric to $\ell_1^{|X|-1}$ then $X$ is a weighted tree and make
a comment on the number of extreme points in the unit ball of
$\tc(M)$.

Section \ref{S:Linfty} is devoted to simple examples of finite
metric spaces, transportation cost spaces on which contain
$\ell_\infty^3$ and $\ell_\infty^4$ isometrically. Earlier, more
complicated finite spaces with this property were provided in
\cite{KMO}. It is an open question whether there exist a finite
metric space $M$ such that $\tc(M)$ contains $\ell_\infty^5$
isometrically.
\section{Preliminaries}
\subsection{Definitions and notation needed for the proofs}

Let us fix some notation for the Laakso graph $\mathcal{L}_n$. We denote the edge, cycle, and cut spaces of $\mathcal{L}_n$  by  $E_n$, $Z_n$ and $C_n$ respectively. The usual $\ell_1$,$\ell_2$, and $\ell_\infty$  norms on $E_n$  are denoted
$\|\cdot\|_1$, $\|\cdot\|_2$, and $\|\cdot\|_\infty$.  The usual inner product is denoted $\langle \cdot, \cdot \rangle$.

The edges of $\mathcal{L}_1$ are  labelled as in Figure~\ref{fig: L1L2}.  We shall fix the reference orientation indicated by the arrows.

For the induction arguments which are used it will  be convenient to label the $6$ sub-$\mathcal{L}_{n-1}$'s of $\mathcal{L}_n$  as $A,\dots,F$ as shown in Figure~\ref{fig: Ln}.
For $n \ge 2$,  the edges of $\mathcal{L}_n$ inherit a reference orientation from $\mathcal{L}_1$  as indicated by the arrows in Figure~\ref{fig: Ln}. The edges of $\mathcal{L}_n$ are oriented from `bottom' to `top' in  Figure~\ref{fig: Ln}.

For each $1 \le j \le n$, we shall use the term `sub-$\mathcal{L}_j$'  to refer to any of the copies of $\mathcal{L}_j$ contained in $\mathcal{L}_n$.

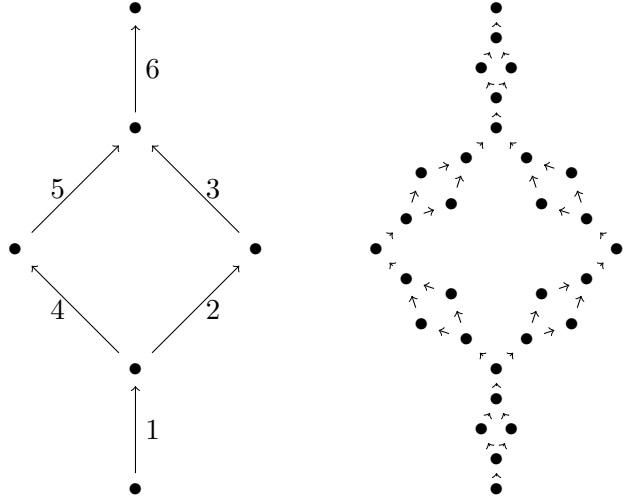
\begin{figure}
\begin{center} {
\begin{tikzpicture}[scale=0.2]

\node (A)  at (16,0) {$\bullet$};
 \node (B)  at (16,8) {$\bullet$};
 \node (C)  at (8,16) {$\bullet$};
 \node (D)  at (24,16) {$\bullet$};
 \node (E)  at (16,24) {$\bullet$};
 \node (F)  at (16,32) {$\bullet$};

\node (A1)  at (-8,0) {$\bullet$};
 \node (B1)  at (-8,8) {$\bullet$};
 \node (C1)  at (-16,16) {$\bullet$};
 \node (D1)  at (0,16) {$\bullet$};
 \node (E1)  at (-8,24) {$\bullet$};
 \node (F1)  at (-8,32) {$\bullet$};

\node [right] at (-8,4) {1};
\node [right] at (-4,12) {2};
\node [right] at (-4,20) {3};
\node [left] at (-12,12) {4};
\node [left] at (-12,20) {5};
\node [right] at (-8,28) {6};

\draw (A1) edge [->] (B1);
\draw (B1) edge  [->]  (C1);
\draw (B1) edge  [->] (D1);
\draw (C1) edge  [->] (E1);
\draw (D1) edge [->]  (E1);
\draw (E1) edge  [->]  (F1);

\node (b1)  at (14,10) {$\bullet$};
 \node (b2)  at (11,11) {$\bullet$};
 \node (b3)  at (13,13) {$\bullet$};
 \node (b4)  at (10,14) {$\bullet$};
\draw (b1) edge  [->] (B);
\draw (b1) edge   [->] (b2);
\draw (b1) edge   [->] (b3);
\draw (b2) edge  [->] (b4);
\draw (b3) edge  [->](b4);
\draw (C) edge  [->](b4);

\node (g1)  at (22,18) {$\bullet$};
 \node (g2)  at (19,19) {$\bullet$};
 \node (g3)  at (21,21) {$\bullet$};
 \node (g4)  at (18,22) {$\bullet$};
\draw (g1) edge  [->] (D);
\draw (g1) edge   [->](g2);
\draw (g1) edge  [->] (g3);
\draw (g2) edge  [->] (g4);
\draw (g3) edge  [->] (g4);
\draw (E) edge  [->] (g4);

 \node (a1)  at (16,2) {$\bullet$};
 \node (a2)  at (15,4) {$\bullet$};
 \node (a3)  at (17,4) {$\bullet$};
 \node (a4)  at (16,6) {$\bullet$};
\draw (a1) edge   [->](A);
\draw (a2) edge [->] (a1);
\draw (a3) edge  [->](a1);
\draw (a4) edge  [->](a2);
\draw (a4) edge  [->](a3);
\draw (B) edge  [->](a4);

 \node (e1)  at (16,26) {$\bullet$};
 \node (e2)  at (15,28) {$\bullet$};
 \node (e3)  at (17,28) {$\bullet$};
 \node (e4)  at (16,30) {$\bullet$};
\draw (e1) edge  [->](E);
\draw (e2) edge  [->] (e1);
\draw (e3) edge  [->](e1);
\draw (e4) edge  [->](e2);
\draw (e4) edge  [->] (e3);
\draw (F) edge  [->] (e4);

 \node (d1)  at (18,10) {$\bullet$};
 \node (d2)  at (21,11) {$\bullet$};
 \node (d3)  at (19,13) {$\bullet$};
 \node (d4)  at (22,14) {$\bullet$};
\draw (d1) edge  [->](B);
\draw (d1) edge  [->] (d2);
\draw (d1) edge  [->](d3);
\draw (d2) edge  [->] (d4);
\draw (d3) edge  [->] (d4);
\draw (D) edge  [->] (d4);

 \node (f1)  at (10,18) {$\bullet$};
 \node (f2)  at (13,19) {$\bullet$};
 \node (f3)  at (11,21) {$\bullet$};
 \node (f4)  at (14,22) {$\bullet$};
\draw (f1) edge  [->] (C);
\draw (f1) edge  [->] (f2);
\draw (f1) edge   [->](f3);
\draw (f2) edge  [->] (f4);
\draw (f3) edge  [->](f4);
\draw (E) edge  [->] (f4);

\end{tikzpicture}} \caption{The Laakso graphs $\mathcal{L}_1$ and $\mathcal{L}_2$} \label{fig: L1L2}
\end{center}
 \end{figure}

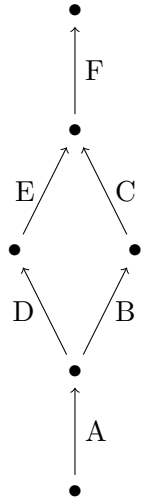
\begin{figure} \begin{center} {
\begin{tikzpicture}[scale=0.20]
 \node (A)  at (16,0) {$\bullet$};
\node (B) at (16,8) {$\bullet$};
\node (C) at (12,16) {$\bullet$};
\node (D) at (20,16) {$\bullet$};
\node (E) at (16,24) {$\bullet$};
\node (F) at (16,32) {$\bullet$};
\draw (A) [->]  edge (B);
\draw (B)  [->]  edge (C); \draw  (B)    [->] edge (D);\draw  (C)  [->]   edge (E); \draw  (D)   [->]  edge (E);
\draw  (E)  [->]  edge (F);
   \node [left] at (14,12) {D};  \node [right] at (18,12) {B};  \node [right] at (18,20) {C};
 \node [left] at (14,20) {E};  \node [right] at (16,4) {A};  \node [right] at (16,28) {F};
\end{tikzpicture}}\caption{The Laakso graph $\mathcal{L}_n$} \label{fig: Ln}\end{center}  \end{figure}
 \subsection{The cycle and cut spaces of $\mathcal{L}_n$} \label{sec: cyclecutspaces}

For each $1 \le j \le n$ and for each given sub-$\mathcal{L}_j$,  $Z_n$ contains the signed indicator function of the outer cycle   (see Figure~\ref{fig: L1L2}) contained in the given sub-$\mathcal{L}_j$. The collection of all  such  signed indicator functions is easily seen to be an algebraic basis of $Z_n$. Counting the total number of sub-$\mathcal{L}_j$'s ,   it  follows that
$\operatorname{dim} Z_n =  (6^n-1)/5$, and hence $\operatorname{dim} C_n = (4\cdot 6^n+1)/5$ since $C_n$ is the orthogonal complement of $Z_n$. However, this basis of $Z_n$  is difficult to work with because it is not  orthogonal.

We shall now construct orthogonal bases for $Z_n$ and $C_n$ which will be used later to analyze projections onto $Z_n$.

\textbf{$n=1$}:   A vector in the edge space will be denoted by a vector
 $$\begin{bmatrix} x_1 & x_2 & x_3 & x_4 & x_5 & x_6 \end{bmatrix},$$
 where $x_i$ denotes the coefficient  on the  edge labelled $i$ (see Figure~\ref{fig: L1L2}).

 Note that $\operatorname{dim} Z_1 =1$ and $\operatorname{dim} C_1 =5$. It is easily seen that
$Z_1$ is spanned by
 \begin{equation} \label{def: h1}h_1=\begin{bmatrix} 0& 1&1& -1& -1& 0\\ \end{bmatrix}.\end{equation}  $C_1$, which is the orthogonal complement of $Z_1$, is easily seen to be  spanned by the row vectors (which are orthogonal) of the following matrix:
\begin{equation} \label{eq: cutvectors} \begin{bmatrix}  -1& 1& 1& 1& 1& -1\\1& 0& 0& 0& 0& -1\\0& 1& -1& 0& 0& 0\\0& 0& 0& 1& -1& 0\\1& 1/2& 1/2& 1/2& 1/2& 1
\end{bmatrix} \end{equation}  Note that these $6$ vectors form an orthogonal basis of $E_1$.

\textbf{$n=2$}:  $\mathcal{L}_2$ is formed from $\mathcal{L}_1$ by replacing each edge of $\mathcal{L}_1$ by a copy of $\mathcal{L}_1$. Similarly, the edge vectors  of $\mathcal{L}_2$ are obtained by replacing  each coefficient $x_i$ of an edge vector of $\mathcal{L}_1$ by the entries of a $6$-dimensional vector.

In this way a vector in $E_1$ generates a vector in $E_2$ according to the following replacement rule:  for each $x \in \mathbb{R}$,
$$ x \mapsto \begin{bmatrix} x& x/2&  x/2& \ x/2&  x/2&   x\\  \end{bmatrix}.$$
We will describe this process of replacement as  `propagation'.

 Define $f_1 \in C_1$  as follows:
$$f_1 =  \begin{bmatrix} 1& 1/2&  1/2& \ 1/2&    1/2& 1\\  \end{bmatrix}.$$

Note that $$f_1 = \frac{1}{2} \begin{bmatrix} 1& 1&  1& \ 0&  0&   1\\  \end{bmatrix}
+ \frac{1}{2} \begin{bmatrix} 1& 0&  0& \  1&  1&   1\\  \end{bmatrix},$$
which expresses $f_1$ as the average of $2$ indicator functions of paths  connecting the bottom vertex of $\mathcal{L}_1$  to the top vertex.  Hence $h_1$   propagates to an average of two signed indicator functions of cycles in $\mathcal{L}_2$. In particular,  $h_1$ propagates to a vector $h_2$  in $Z_2$.

In addition to this vector, each of the $6$   copies of $\mathcal{L}_1$  supports  a `new' cycle vector   given  by
 $$\begin{bmatrix} 0& 1&1& -1& -1& 0\\ \end{bmatrix}.$$
(Its coefficients on the other five copies of $\mathcal{L}_1$  are all zero.) Note that this vector is orthogonal to  the propagated vector since it is orthogonal  to $f_1$.

The $5$ basis vectors of $C_1$ propagate to form basis vectors of $C_2$. In addition, supported on each of the six copies of $\mathcal{L}_1$  we obtain $4$  `new' orthogonal cut vectors given by the row vectors of the following matrix:

$$\begin{bmatrix}  -1& 1& 1& 1& 1& -1\\1& 0& 0& 0& 0& -1\\0& 1& -1& 0& 0& 0\\0& 0& 0& 1& -1& 0\\
\end{bmatrix}$$ Note that the row vectors are orthogonal to   $f_1$. Hence the new cut vectors are orthogonal to the propagated cut vectors. The $5$ propagated cut vectors and the $24$ new cut vectors together form an orthogonal basis of
the cut space $C_2$.

\textbf{$n \ge 3$}: This is similar to the case $n=2$. The orthogonal bases of $Z_{n-1}$ and $C_{n-1}$ propagate to  collections  of orthogonal vectors in $Z_n$ and $C_n$. In addition,  each
of the $6^{n-1}$ copies of $\mathcal{L}_1$ supports one new cycle vector and 4 new cut vectors as above.

Let us check  these  claims.   The claimed bases of $Z_n$ and $C_n$ are orthogonal  and have the correct cardinality. So it suffices to check they they are contained  in $Z_n$ and $C_n$ respectively. For $n \ge2$, let  $h_n$ be the propagation of $h_{n-1}$ and let $f_n$ be the propagation of $f_{n-1}$ (see Figure~\ref{fig: fn}).
It suffices to check that $h_n \in Z_n$.
A straightfoward induction shows that $f_n$ is the average of $2^n$ indicator functions of paths joining the bottom and top vertices of $\mathcal{L}_n$.
Hence (see  Figure~\ref{fig: gnhn}) $h_n$ is the average of $2^{n-1}$ signed indicator functions of large cycles in $\mathcal{L}_n$. In particular, $h_n \in Z_n$ as desired.

Recalling that $C_n$ is the orthogonal complement of $Z_n$, the orthogonality of the basis guarantees that the claimed basis of $C_n$ is indeed contained in $C_n$.
\subsection{Norms of cycle and cut vectors} \label{sec: normcalcs}
Note that $$\|f_1\|_1 = 4, \|f_1\|_2^2 = 3.$$
For $n \ge 2$, define $f_n \in E_n$ inductively as shown in  Figure~\ref{fig: fn}.
\begin{figure}\begin{center} {
\begin{tikzpicture}[scale=0.20]
 \node (A)  at (16,0) {$\bullet$};
\node (B) at (16,8) {$\bullet$};
\node (C) at (12,16) {$\bullet$};
\node (D) at (20,16) {$\bullet$};
\node (E) at (16,24) {$\bullet$};
\node (F) at (16,32) {$\bullet$};
\draw (A)  edge (B);
\draw (B)  edge (C); \draw  (B)  edge (D);\draw  (C)  edge (E); \draw  (D)  edge (E);
\draw  (E)  edge (F);
   \node [left] at (14,12) {$\frac{1}{2}f_{n-1}$};  \node [right] at (18,12) {$\frac{1}{2}f_{n-1}$};  \node [right] at (18,20) {$\frac{1}{2}f_{n-1}$};
 \node [left] at (14,20) {$\frac{1}{2}f_{n-1}$};  \node [right] at (16,4) {$f_{n-1}$};  \node [right] at (16,28) {$f_{n-1}$}; \node at (16,16) {$f_n$};
\end{tikzpicture}} \caption{$f_n$ defined on each copy of $\mathcal{L}_{n-1}$ in  $\mathcal{L}_n$} \label{fig: fn} \end{center} \end{figure}
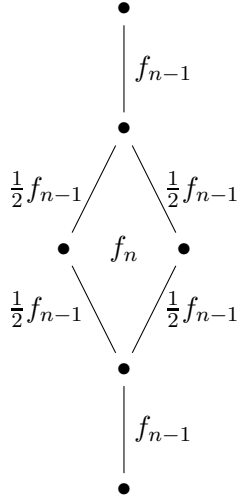
Note that
$$ \|f_n\|_1= 4\|f_{n-1}\|_1 = 4^n, \|f_n\|_2^2 = 3\|f_{n-1}\|_2^2 = 3^n.$$

Recall from \eqref{def: h1}  that  $h_1 \in Z_1$ was defined by  $$h_1 =  \begin{bmatrix} 0& 1&  1& \ -1&    -1& 0\\  \end{bmatrix}.$$
Now define $g_1 \in C_1$ by  $$g_1 =  \begin{bmatrix} -1& 1&  1& \ 1&    1& -1\\  \end{bmatrix},$$
and, for $n\ge2$,  define $g_n\in C_n$ and $h_n\in Z_n$ inductively  as shown in Figure~\ref{fig: gnhn}.
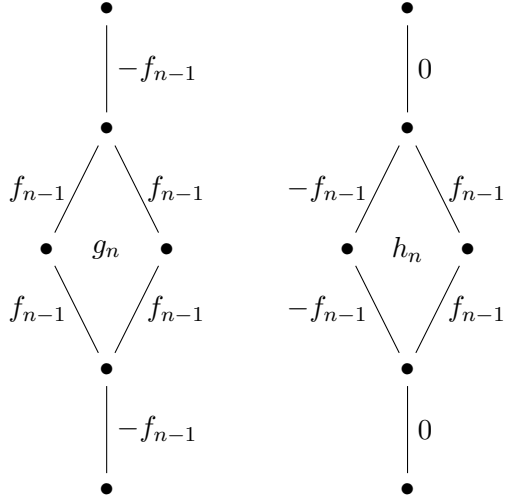
\begin{figure}\begin{center} {
\begin{tikzpicture}[scale=0.20]
 \node (A)  at (16,0) {$\bullet$};  \node (A1)  at (36,0) {$\bullet$};
\node (B) at (16,8) {$\bullet$};  \node (B1) at (36,8) {$\bullet$};
\node (C) at (12,16) {$\bullet$}; \node (C1) at (32,16) {$\bullet$};
\node (D) at (20,16) {$\bullet$}; \node (D1) at (40,16) {$\bullet$};
\node (E) at (16,24) {$\bullet$}; \node (E1) at (36,24) {$\bullet$};
\node (F) at (16,32) {$\bullet$}; \node (F1) at (36,32) {$\bullet$};
\draw (A)  edge (B); \draw (A1)  edge (B1);
\draw (B)  edge (C); \draw (B1)  edge (C1);
 \draw  (B)  edge (D);  \draw  (B1)  edge (D1);
\draw  (C)  edge (E); \draw  (C1)  edge (E1);
 \draw  (D)  edge (E); \draw  (D1)  edge (E1);
\draw  (E)  edge (F); \draw  (E1)  edge (F1);
   \node [left] at (14,12) {$f_{n-1}$};  \node [right] at (18,12) {$f_{n-1}$};  \node [right] at (18,20) {$f_{n-1}$};
 \node [left] at (14,20) {$f_{n-1}$};  \node [right] at (16,4) {$-f_{n-1}$};  \node [right] at (16,28) {$-f_{n-1}$}; \node at (16,16) {$g_n$};
\node [left] at (34,12) {$-f_{n-1}$};  \node [right] at (38,12) {$f_{n-1}$};  \node [right] at (38,20) {$f_{n-1}$};  \node at (36,16) {$h_n$};
 \node [left] at (34,20) {$-f_{n-1}$};  \node [right] at (36,4) {$0$};  \node [right] at (36,28) {$0$};
\end{tikzpicture}} \caption{$g_n$  and $h_n$ defined on each copy of $\mathcal{L}_{n-1}$ in  $\mathcal{L}_n$} \label{fig: gnhn} \end{center} \end{figure}
Note that $h_n$ is the cycle vector obtained from  $h_1$ by repeated propagation,   $g_n$ is the cut vector obtained from
$g_1$ by repeated propagation,
$$\|g_n\|_1 = 6\|f_{n-1}\|_1 = \frac{3}{2} 4^n, \|g_n\|_2^2 = 6\|f_{n-1}\|_2^2 = 2 \cdot 3^n,$$
and $$\|h_n\|_1 = 4\|f_{n-1}\|_1 =  4^n, \|h_n\|_2^2 = 4\|f_{n-1}\|_2^2 = \frac{4}{3} \cdot 3^n.$$
Hence, in particular,  \begin{equation} \label{eq: normestimates} \frac{\|g_n\|_1}{\|g_n\|_2^2} = \frac{\|h_n\|_1}{\|h_n\|_2^2}= (\frac{4}{3})^{n-1}.
\end{equation}

 Note that each sub-$\mathcal{L}_j$
  supports a unique $Z_n$ basis vector $H_j$ of the form $h_j$  and a unique $C_n$ basis vector $G_j$ of the form $g_j$. To justify this claim,  let  $L_j$ be a sub-$\mathcal{L}_j$ of $\mathcal{L}_{n}$.
  For $j=1$, $G_1$ and $H_1$ are the `new'   $g_1$
  and $h_1$ basis vectors supported on $L_1$ arising   in  the passage from $Z_{n-1}$ to $Z_n$ and $C_{n-1}$ to $C_n$  described above. For $j>1$, note that $L_j$
  evolves from a unique sub-$\mathcal{L}_1$ of  $\mathcal{L}_{n-1-j}$, $L'_1$ say. Let $G'_1$ and $H'_1$ be the $g_1$ and $h_1$ basis vectors supported on $L_1'$.  Propagating  $G'_1$ and $H'_1$ repeatedly  $(j-1)$ times produces basis vectors $G_j$ and $H_j$  of the form $g_j$ and $h_j$  that are supported on $L_j$ as claimed.

The next two lemmas will be used in Section~\ref{sec: projectionPn}. 
\begin{lemma}  \label{lem: gnhnsigns} Let $1 \le j \le n$ and let $H_j$ and $G_j$ be supported in some sub-$\mathcal{L}_j$, $L_j$, say. Then, for every edge vector $e$ belonging to $L_j$, we have  \begin{enumerate}  \item $\langle e, H_j \rangle = 0 \Leftrightarrow \langle e, G_j \rangle <0$.
\item If $\langle e, H_j \rangle \ne 0$  then  $\langle e, G_j \rangle >0$ and  $|\langle e, H_j \rangle|  =\langle e, G_j \rangle$.
\end{enumerate}
\end{lemma} \begin{proof} (1) From   Figure~\ref{fig: gnhn}, note  that $\langle e, H_j\rangle =0$ if and only if $e$ belongs to the $A$ or $F$ sub-$\mathcal{L}_{j-1}$
of $L_j$  if and only if
$\langle e,G_j\rangle < 0$.

(2) If $\langle e, H_j \rangle \ne 0$ then $e$ belongs to the $B,C,D$ or $E$ sub-$\mathcal{L}_{j-1}$.  From  Figure~\ref{fig: gnhn}, note that $\langle e, G_j \rangle >0$ and $ |\langle e, H_j\rangle|=\langle e, G_j\rangle$.
\end{proof}
To state the next lemma,  let us first fix some notation. For $1 \le j \le n$, let $L_j$ be a sub-$\mathcal{L}_j$ of $\mathcal{L}_n$ such that $(L_j)_{j=1}^n$
is an increasing chain, i.e., $L_1 \subset L_2 \subset \dots \subset L_n = \mathcal{L}_n$. Let $S_j$ be the set of edge vectors contained in $L_j$, so that
$(S_j)_{j=1}^n$ is also increasing. Finally, for $1 \le j \le n$,  let $G_j$ and $H_j$ be the cut and cycle basis vectors corresponding to $L_j$ (of the form $g_j$ and $h_j$).
\begin{lemma} \label{lem: gjhjinnerproducts} Let $1 \le j \le n$. Then for every $e \in S_1$, we have
$$ \langle e, G_j \rangle = (\frac12)^{\alpha_j}\operatorname{sgn}(\langle e, G_j \rangle)\quad\text{and}\quad \langle e, H_j \rangle = (\frac12)^{\alpha_j}\operatorname{sgn}(\langle e, H_j \rangle),
$$ where $\alpha_1 = 0$ and, for $j \ge 2$, $\alpha_j$ is the cardinality of the set $\{1 \le r < j \colon S_{r-1} \subset \operatorname{supp}(H_r)\}$
(here $\operatorname{sgn}(0)=0$).
 \end{lemma} \begin{proof}  The result clearly holds for $j=1$. So suppose that the result holds for $j= j_0$, where $1 \le j_0 < n$.
 For $1 \le j \le n$, let $F_j$ be the vector of the form $f_j$ corresponding to $L_j$.  From Figure~\ref{fig: gnhn}, we have $$|\langle e, G_{j_0+1} \rangle| = \langle e, F_{j_0} \rangle.$$
 If $S_{j_0-1} \subset \operatorname{supp}(H_{j_0})$, then $\alpha_{j_0+1} = \alpha_{j_0}+1$ and, from Figure~\ref{fig: fn},
 $$ \langle e, F_{j_0} \rangle = \frac{1}{2}  \langle e, F_{j_0-1} \rangle = \frac{1}{2}| \langle e, G_{j_0} \rangle|,$$
(where $\langle e, F_0 \rangle = 1$ by convention in the case $j_0=1$).
 So by the inductive hypothesis,
  $$\langle e, G_{j_0+1} \rangle = \frac{1}{2}|\langle e, G_{j_0} \rangle| \operatorname{sgn}(\langle e, G_{j_0+1} \rangle) = 
  (\frac{1}{2})^{\alpha_{j_0+1}}\operatorname{sgn}(\langle e, G_{j_0+1} \rangle) $$ as desired. On the other hand, if $S_{j_0-1}$ is disjoint from
  $\operatorname{supp}(H_{j_0})$, then $\alpha_{j_0+1}=\alpha_{j_0}$ and from Figure~\ref{fig: fn},
  $$ \langle e, F_{j_0} \rangle =   \langle e, F_{j_0-1} \rangle = | \langle e, G_{j_0} \rangle|.$$ So by the inductive hypothesis,
   $$\langle e, G_{j_0+1} \rangle = |\langle e, G_{j_0} \rangle| \operatorname{sgn}(\langle e, G_{j_0+1} \rangle) = (\frac{1}{2})^{\alpha_{j_0+1}}\operatorname{sgn}(\langle e, G_{j_0+1} \rangle) $$ as desired. The  stated result for $\langle e, H_j \rangle$ follows from the result for
    $\langle e, G_j \rangle$ and  Lemma~ \ref{lem: gnhnsigns}.
 \end{proof}
\section{A projection onto the cycle space}  \label{sec: projectionPn}

In this section we define a projection $P_n$ from $E_n$ onto its cycle space $Z_n$ which has relatively small (linear in $n$, i.e., logarithmic in $\operatorname{dim}(E_n)$) norm on $(E_n,\|\cdot\|_1)$.

Let us first observe that the \textit{orthogonal} projection $\overline{P}_n$  of $E_n$ onto $Z_n$ has large (exponential in $n$)
norm on  $(E_n,\|\cdot\|_1)$.

\begin{proposition} $$\|\overline{P}_n\|_1 \ge  (\frac43)^{n-1}.$$
\end{proposition}
\begin{proof} Let $e$ be the edge vector in $Z_n$  corresponding to the  `lowest' edge (with respect to the `bottom' to `top' orientation) in the  sub-$\mathcal{L}_{n-1}$ labelled as $B$. Then
$\langle e, h_n \rangle =1$ and $\langle e, h \rangle =0$ if $h \ne h_n$ is any other  basis vector of $Z_n$.  Hence, using  \eqref{eq: normestimates},
$$\|\overline{P}_n\|_1 \ge \|\overline{P}_n(e)\|_1 = \langle e, h_n  \rangle \frac{\|h_n\|_1}{\|h_n\|_2^2} = (\frac43)^{n-1}.$$   \end{proof}
The definition of $P_n$  is inductive.
 $P_1$ is the orthogonal projection.

  Suppose $n \ge 2$.  We start the definition of $P_n$ by setting  $P_n(g_n)=0$ and $P_n(g)=0$ for every cut vector $g$ in the orthogonal  basis of  $C_n$  which is \textit{not} of the form $g_j$
 for some sub-$\mathcal{L}_j$ ($1 \le j \le n-1$). This is to be expected as we shall show in the next section that this holds for any projection which is invariant with respect to a natural group of isometries of $E_n$.
 Thus, to complete the definition,  it suffices to define $P_n(g_j)$ for each sub-$\mathcal{L}_j$.

We shall label the six sub-$\mathcal{L}_{n-1}$'s  as $A,\dots,F$ as shown in Figure~\ref{fig: Ln}.
On  $A$ and $F$ we define $P_n$ to be a copy of $P_{n-1}$.
So it suffices to  define $P_n(g_j)$ for all $g_j$  supported on a  a sub-$\mathcal{L}_j$ contained in  $B,C,D$ or $E$.  The definition of $P_n(g_j)$  will proceed \textit{backwards} from $j = n-1$ to $j=1$.

Let $S_{n-1}$ be the set of edge vectors of any one of $B,C,D$ or $E$.  Now let $S_{n-2}$ be the set of edge vectors of  any one of the $6$  sub-$\mathcal{L}_{n-2}$'s supported in
$S_{n-1}$. Continue in this way to obtain a chain $S_{n-1} \supset S_{n-2} \supset \dots \supset S_1$. Finally, let   $e$ be one of the $6$ edge vectors
contained
in $S_1$.
Note that  $S_1$   uniquely determines the chain $(S_j)_{j=1}^{n-1}$ and that every edge vector $e$ in the support of $B,C,D$, or $E$ determines
a unique  choice of $S_1$.

For  each $1 \le j \le n-1$, let  $G_j$   denote the $g_j$  cut vector and let $H_j$ denote  the $h_j$ cycle vector corresponding to  the sub-$\mathcal{L}_{j}$ supported on
$S_j$.
We shall define $P_n(G_j)$ inductively along the chain $(S_j)_{j=1}^{n-1}$  starting with $j = n-1$.  By varying the chain we   define  $P_n(G_j)$
for every  cut vector in the orthogonal basis of $C_n$ which is of the form $G_j$ for some sub-$\mathcal{L}_j$ ($1 \le j \le n-1$). Since each sub-$\mathcal{L}_j$ occurs in several different chains, we must also  check   that  $P_n(G_j)$  is  well-defined.

The motivating idea behind this definition is a `balancing' of certain norms which is described
in (iv) below. However,  since the proof  is  lengthy and not particularly intuitive, we will describe the strategy before going into the details.  The definition of
 $P_n(G_j)$  will involve a sequence of vectors  $(X_j)_{j=1}^n$ and sequences of  scalars  $(x_j)_{j=1}^n$ and $(a_j)_{j=1}^n$, which are defined inductively. The strategy behind the definition of $P_n$ and  the  proof of Theorem~\ref{thm: goodprojection} below is as follows:
 \begin{itemize}
 \item[(i)] $X_{j}$ is completely determined by $S_{j-1}$ and is defined inductively as  a linear combination of  $H_{j}, H_{j+1}, \dots, h_n$.
 \item[(ii)] The definition of $X_{j}$ given by  \eqref{eq: defofX_j} has two cases,   depending on  whether or not  $S_{j-1}$ is contained in the support of $H_j$ (equivalently,  whether or not $e \in \operatorname{supp}(H_j)$).
 \item[(iii)] The choice of  $a_j$ as  defined by  \eqref{eq: defnofa_j} ensures that   $X_j$ has  roughly the same
 $\|\cdot\|_1$ norm in both cases. \item[(iv)]    Hence
  $P_n(G_j)$, as defined by  \eqref{eq: defofa_jX_j}, has  roughly the same norm  in both cases of the definition of $X_{j+1}$. It is this   balancing which
  ultimately  leads  to a  projection of relatively small norm.  (Note also that $P_n(G_j)$ is a certain linear combination of $H_{j+1},H_{j+2},\dots, h_n$.)
 \item[(v)] The choice of $a_j$ ensures that  $\|X_j\|_1 \le x_j := (1-a_j) x_{j+1}$.  \item[(vi)]
 It is shown in Lemma~\ref{lem: imageofe}   that $X_1 = P_n(e)$, and hence $\|P_n(e)\|_1 \le x_1$. This is the key
  estimate  in  the proof of Theorem~\ref{thm: goodprojection}.
 \item[(vii)] $(x_j)_{j=1}^n$ satisfies a recurrence relation which is solved in Lemma~\ref{lem: recurrence}. This leads to the estimate $\|P_n\|_1 \le (n+1)/2$, which is proved in Theorem~\ref{thm: goodprojection}.
 \end{itemize}

Let us now go through the details of the  definition of $P_n(G_j)$ starting with $j = n-1$. Set \begin{equation} \label{eq: defofX_n}
 X_n = \operatorname{sgn}(\langle e, h_n \rangle)  \frac{h_n}{\|h_n\|_2^2} \quad\text{and}\quad x_n = \|X_n\|_1 = (\frac43)^{n-1}, \end{equation}
where $\operatorname{sgn}(a)$ is the sign of $a$.  Define
$$P_n(\frac{G_{n-1}}{\|G_{n-1}\|_2^2}) = a_{n-1}X_n,$$
where $a_{n-1}$ is defined by the equation $$(1-a_{n-1}) (\frac43)^{n-1} = (\frac12 + a_{n-1})(\frac43)^{n-1}+ (\frac43)^{n-2}.$$
(Note that, in fact, $a_{n-1} = -1/8$.)  Now set $$
X_{n-1}= \begin{cases}  (\frac12 + a_{n-1})X_n +  \operatorname{sgn}(\langle e, H_{n-1} \rangle)  \frac{H_{n-1}}{\|H_{n-1}\|_2^2}, &e \in \operatorname{supp}(H_{n-1}),\\ (1-a_{n-1})X_n, &e \notin \operatorname{supp}(H_{n-1}).\end{cases}$$
Since  $1 - a_{n-1}  =9/8>0$ and $\dfrac12 + a_{n-1}=3/8 >0$, the triangle inequality and \eqref{eq: normestimates} give
\begin{align*} \|X_{n-1}\|_1 &\le [ (\frac12 + a_{n-1})\|X_n\|_1 +  \frac{\|H_{n-1}\|_1}{\|H_{n-1}\|_2^2}] \vee (1-a_{n-1})\|X_n\|_1\\
&=[ (\frac12 + a_{n-1})(\frac43)^{n-1}+ (\frac43)^{n-2}] \vee (1-a_{n-1}) (\frac43)^{n-1} \\
&=(1-a_{n-1}) (\frac43)^{n-1} \\&= (1-a_{n-1}) \|X_n\|_1.
\end{align*} Set $x_{n-1} =(1-a_{n-1}) \|X_n\|_1$. Then   $\|X_{n-1}\|_1 \le x_{n-1}$.

Let us now turn to the inductive step, which is similar to the case $j=n-1$.  Suppose that $1 \le j < n-1$ and that $X_{j+1}$, $x_{j+1}$, and $P_n(G_{j+1})$ have been defined with $\|X_{j+1}\|_1 \le x_{j+1}$. Now define
 \begin{equation} \label{eq: defofa_jX_j} P_n(\frac{G_{j}}{\|G_{j}\|_2^2}) = a_{j}X_{j+1}, \end{equation}  where $a_j$ is defined by the equation
 \begin{equation} \label{eq: defnofa_j}
(1 - a_j)x_{j+1} = (\frac12 + a_j) x_{j+1} + (\frac43)^{j-1}. \end{equation} Set \begin{equation} \label{eq: defofX_j}
X_{j}= \begin{cases}  (\frac12 + a_{j})X_{j+1} +  \operatorname{sgn}(\langle e, H_{j} \rangle)  \frac{H_{j}}{\|H_{j}\|_2^2}, &e \in \operatorname{supp}(H_{j}),\\ (1-a_{j})X_{j+1}, &e \notin \operatorname{supp}(H_{j}).\end{cases}. \end{equation}  
It is worth observing that, for $j \ge 2$,  $X_j$ does not  depend on the particular choice of $e$ from $S_1$. Hence, for $j \ge 1$,  $P_n(G_j)$
defined by \eqref{eq: defofa_jX_j} is also
independent of the choice of $e$ as required.  But we prove below (Lemma~\ref{lem: imageofe})  that $X_1 = P_n(e)$, which does depend on the choice of $e$.

We prove in Lemma~\ref{lem: recurrence} below that $\dfrac12 + a_j >0$ and $1-a_j>0$. Hence, by the triangle inequality and \eqref{eq: normestimates},
\begin{align*} \|X_{j}\|_1 &\le [ (\frac12 + a_{j})\|X_{j+1}\|_1 +  \frac{\|H_{j}\|_1}{\|H_{j}\|_2^2}] \vee (1-a_{j})\|X_{j+1}\|_1\\
&\le [ (\frac12 + a_{j})x_{j+1}+ (\frac43)^{j-1}] \vee (1-a_{j}) x_{j+1}\\
&=(1-a_{j}) x_{j+1}.
\end{align*} Finally, set $x_j = (1-a_j) x_{j+1}$ to complete the inductive step.

To check that $P_n(G_j)$ as given by  \eqref{eq: defofa_jX_j} is well-defined, we need to check that it depends only on $\operatorname{supp}(G_j)=S_j$. To see this, note that $S_j$ determines its `ancestors' $S_{j+1},\dots, S_{n-1}$ uniquely.
Moreover,  the definition of  $X_{j+1}$ (see \eqref{eq: defofX_j} and replace $j$ by $j+1$) actually depends only on $S_j$ since  $\operatorname{sgn}(\langle e, H_{j+1} \rangle) $
is simply  the (constant) sign of $H_{j+1}$ on $S_j$. Hence $P_n(G_j)$ is indeed well-defined.

By considering every chain $S_{n-1}\supset S_{n-2}\supset\dots \supset S_1$, we  define $S(g)$ for every cut vector  $g$ of the form $g_j$ for  some  sub-$\mathcal{L}_j$.

The definition of $P_n$ is now complete. (Recall that we started the definition by  setting  $P_n(g_n)=0$ and   $P_n(g)=0$ for all other cut vectors $g$  in the orthogonal basis of $C_n$  described above.)

\begin{lemma} \label{lem: imageofe} $P_n(e) = X_1$. \end{lemma}
\begin{proof} Using Lemma~\ref{lem: gnhnsigns} and the fact (see \eqref{eq: defofa_jX_j}) that
 $$P_n(\frac{G_{j}}{\|G_{j}\|_2^2}) = a_{j}X_{j+1},$$ we can combine the two cases in the definition \eqref{eq: defofX_j} of $X_j$  as follows:
$$X_j = (\frac12)^{\varepsilon_j} X_{j+1} +\operatorname{sgn}(\langle e, G_{j} \rangle)  \frac{P_n(G_{j})}{\|G_{j}\|_2^2} +
 \operatorname{sgn}(\langle e, H_{j} \rangle)  \frac{H_{j}}{\|H_{j}\|_2^2},$$
where $$\varepsilon_j = \begin{cases} 1, &S_{j-1} \subset \operatorname{supp}(H_j),\\
0,   &S_{j-1} \cap \operatorname{supp}(H_j) = \emptyset \end{cases}$$
and setting  $\operatorname{sgn}(0)=0$. After  repeated application of this formula, starting at $j=1$ and ending at $j=n-1$, and  then substituting
(see \eqref{eq: defofX_n})
$$X_n = \operatorname{sgn}(\langle e, h_n \rangle)  \frac{h_n}{\|h_n\|_2^2},$$
we obtain
$$X_1 = \sum_{j=1}^{n-1} (\frac12)^{\alpha_j}[\operatorname{sgn}(\langle e, G_{j} \rangle)  \frac{P_n(G_{j})}{\|G_{j}\|_2^2} +
\operatorname{sgn}(\langle e, H_{j} \rangle)  \frac{H_{j}}{\|H_{j}\|_2^2}] + (\frac12)^{\alpha_n}\operatorname{sgn}(\langle e, h_{n} \rangle)  \frac{h_{n}}{\|h_{n}\|_2^2}, $$ where $\alpha_1=0$ and, for $j\ge2$,
$\alpha_j$ is the cardinality of the set $\{1 \le r < j \colon S_{r-1} \subset \operatorname{supp}(H_r)\}$.   By Lemma~\ref{lem: gjhjinnerproducts}, for $1 \le j \le n-1$,
$$\langle e, G_{j} \rangle= (\frac12)^{\alpha_j}\operatorname{sgn}(\langle e, G_{j} \rangle)$$
and  $$\langle e, H_{j} \rangle= (\frac12)^{\alpha_j}\operatorname{sgn}(\langle e, H_{j} \rangle)$$
and $$\langle e, h_{n} \rangle= (\frac12)^{\alpha_n}\operatorname{sgn}(\langle e, h_{n} \rangle).$$ Hence
\begin{align*} X_1 &= \sum_{j=1}^{n-1} [\langle e, G_{j} \rangle  \frac{P_n(G_{j})}{\|G_{j}\|_2^2} +\langle e, H_{j} \rangle  \frac{H_{j}}{\|H_{j}\|_2^2}] + \langle e, h_{n} \rangle \frac{h_{n}}{\|h_{n}\|_2^2}\\
&=P_n([ \sum_{j=1}^{n-1} \langle e,\frac{G_{j}}{\|G_{j}\|_2}\rangle \frac{G_{j}}{\|G_{j}\|_2} + \langle e,\frac{H_{j}}{\|H_{j}\|_2}\rangle \frac{H_{j}}{\|H_{j}\|_2}]+ \langle e,\frac{h_n}{\|h_n\|_2}\rangle \frac{h_n}{\|h_n\|_2})\\
&= P_n(e).
 \end{align*} 
 To see the last line of the above, note that if $\langle e,g \rangle \ne 0$, for $g$ belonging to the othogonal basis of $C_n$, then  either $g = G_j$ for some $1 \le j \le n-1$ or $P_n(g)=0$. This is because  we began the definition of $P_n$ by setting  $P_n(g_n)=0$ and $P_n(g)=0$ for every cut vector $g$ in the orthogonal  basis of  $C_n$  which is \textit{not} of the form $g_j$ for some sub-$\mathcal{L}_j$. On the other hand, if $g$ is of the form   $g_j$  and  $\langle e, g \rangle \ne 0$ then  $ \operatorname{supp}(g) = S_j$,  i.e.,   $g = G_j$.
 Similarly,
  if $\langle e, h \rangle \ne 0$, for $h$ belonging  to the orthogonal basis of $Z_n$, then either $h = H_j$ or $h=h_n$. So the above  expression for $X_1$ is simply $P_n$ applied to the expansion of $e$ with respect to the othogonal basis of $E_n$.

 \end{proof} \begin{lemma} \label{lem: recurrence} $x_1 = \dfrac{n+1}{2}$ and  $\min(1-a_j, \frac12 + a_j)>0$ for  $1 \le j < n-1$. \end{lemma}
\begin{proof} Recall that $x_n =(4/3)^{n-1}$ (see \eqref{eq: defofX_n}) and that, for $1 \le j \le n-1$, $x_j$ satisfies the recurrence
$$x_j =  (1-a_j)x_{j+1} = (\frac12 + a_j)x_{j+1} + (\frac43)^{j-1}$$
which serves to define $a_j$ for $1 \le j \le n-1$. Hence
  \begin{align*} x_j &= \frac12[(1-a_j)x_{j+1} + (\frac12 + a_j)x_{j+1} + (\frac43)^{j-1}]\\
&= \frac34 x_{j+1} + \frac12(\frac43)^{j-1}.
\end{align*} The solution to this recurrence is $$x_j  = \frac{n+2-j}{2} (\frac43)^{j-1}.$$ Note that \begin{align*}
a_jx_{j+1}= x_{j+1}-x_j
=(\frac43)^j[-\frac14 + \frac{n-j}{8}].\end{align*}
Hence  $a_{n-1} = -\dfrac{1}{8},  a_{n-2}=0$, and $0 < a_j<1$ for $1 \le j \le n-3$. In all cases $\min(1-a_j, \dfrac12 + a_j)>0$.
\end{proof} \begin{theorem} \label{thm: goodprojection} $\|P_n\|_1 \le  \dfrac{n+1}{2}$.
\end{theorem} \begin{proof} Recall that $P_1$  is the orthogonal projection onto $Z_1$:
$$P_1(e_i)  = \begin{cases} \frac{\pm 1}{4}(e_2+e_3-e_4-e_5), &i=2,3,4,5\\
0, &i=1,6. \end{cases}$$  Clearly, $\|P_1\|_1 = 1$. Now suppose $n \ge 2$. If $e$ is an edge vector belonging to the $A$ or $F$ sub-$\mathcal{L}_{n-1}$,
then, by the inductive hypothesis,  $$\|P_n(e)\| _1\le \|P_{n-1}\|_1\le \frac{n}{2}.$$ On the other hand, if $e$ belongs to the $B,C,D$ or $E$ sub-$\mathcal{L}_{n-1}$, then
$P_n(e) = X_1$  for the chain $(S_j)_{j=1}^{n-1}$ with $e \in S_1$, so by Lemma~\ref{lem: recurrence},
$$\|P_n(e)\|_1 = \|X_1\|_1 \le x_1 = \frac{n+1}{2}.$$
Hence $$\|P_n\|_1 = \max_e \|P_n(e)\|_1 \le \frac{n+1}{2}.$$
\end{proof} \section{Invariant Projections} \label{sec: invariant}
In this section we prove that the projection $P_n$  constructed in the previous section is close to being optimal. First we show that we may restrict attention to projections that are `invariant' with respect to a certain group of isometries of  $E_n$. Then we show that $P_n$ is close to being optimal in the sense that its operator norm is of the same order.

 First, let us define a group of isometries of $(E_n,\|\cdot\|_2)$. To that end,
 let us say that a cut vector $g$ belonging to the orthogonal basis of the cut space $C_n$ is \textbf{special}
if $g$ is of the the form $g_j$ for some sub-$\mathcal{L}_j$ for $1 \le j \le n$. We shall say that $g$ is \textbf{non-special}  if $g$ is not of the form $g_j$ and $g$ is not the unique cut vector  propagated by  $\begin{bmatrix} 1& 1/2&1/2& 1/2& 1/2& 1\\\end{bmatrix}$.

If $g$ is a non-special cut vector then there will be a smallest sub-$\mathcal{L}_j$  ($1 \le j \le n$) which contains its support. Let us call this the \textit{support sub-$\mathcal{L}_j$} of $g$. Let $\psi_g$ be the natural  isometry of
$E_n$ induced by interchanging  $\{g>0\}$ and $\{g<0\}$.  Since there are three types of non-special vector,  namely those  cut vectors propagated by the second, third, and fourth rows of $\eqref{eq: cutvectors}$,  $\psi_g$ is effectuated by  either
(a) interchanging the $B$ and $C$ sub-$\mathcal{L}_{j-1}$ of its support (using  the inductively defined isomorphism between $B$  and  $C$
and $\mathcal{L}_{j-1}$), or (b) interchanging the $D$ and $E$ sub-$\mathcal{L}_{j-1}$, or (c) interchanging the $A$ and $F$ sub-$\mathcal{L}_{j-1}$.
Note that $Z_n$ and $C_n$ are $\psi_g$-invariant subspaces.

Similarly,  each $Z_n$ basis vector  $h$  has a  support sub-$\mathcal{L}_j$. Let  $\phi_h$ be the natural isometry induced by interchanging $\{h>0\}$  and $\{h <0\}$. Then $\phi_h$ is effectuated by interchanging the $B$ and $E$ sub-$\mathcal{L}_{j-1}$ and the $C$ and $D$ sub-$\mathcal{L}_{j-1}$ of the support sub-$\mathcal{L}_j$ of $h$.
Note that $Z_n$ and $C_n$ are $\phi_h$-invariant subspaces.

Note that $\phi_h^* = \phi_h= \phi^{-1}$ and $\psi_g^* = \psi_g=\psi_g^{-1}$ when considered as isometries of the Euclidean space $(E_n,\|\cdot\|_2)$.

Let  $G$ be the (finite) group generated by the collection of  all  $\psi_g$ and $\phi_h$ isometries. Let $Q$ be any projection form $E_n$ onto $Z_n$. Then
$$P = \frac{1}{|G|}\sum_{\theta \in G} \theta^{-1}Q\theta$$ satisfies
$\|P\|_1 \le \|Q\|_1$, and $P\theta = \theta P$ for all $\theta \in G$. Moreover, $P$ is also a projection onto $Z_n$ since $Z_n$ and $C_n$ are $\theta$-invariant for each $\theta \in G$. \begin{lemma} \label{lem: nonspecialimage}  If $g$ is \textit{non-special} or if $g$ is the (unique) cut vector propagated by $\begin{bmatrix} 1& 1/2&1/2& 1/2& 1/2& 1\\\end{bmatrix}$ then $P(g)=0$.
\end{lemma} \begin{proof} Since $P(g) \in Z_n$ it suffices to show that $\langle P(g), h \rangle = 0$ for every  $h$ belonging to the basis of $Z_n$.
 If $\operatorname{supp}(g) \subseteq \operatorname{supp}(h)$ then $\psi_g(g)=-g$ and $\psi_g(h)=h$. So
$$\langle P(g),h \rangle =\langle P(g),\psi_g(h) \rangle =\langle \psi_g(P(g)),h \rangle =
\langle P(\psi_g(g)),h \rangle =-\langle P(g),h \rangle.$$

On the other hand, if $\operatorname{supp}(h) \subseteq \operatorname{supp}(g)$  or $\operatorname{supp}(h) \cap \operatorname{supp}(g)=\emptyset$
then $\phi_h(h)=-h$ and $\phi_h(g)=g$. So $$\langle P(g),h \rangle =\langle P(\phi_h(g)),h \rangle =\langle \phi_h(P(g)),h \rangle =
\langle P(g),\phi_h(h) \rangle =-\langle P(g),h \rangle.$$

Hence, in both cases, $\langle P(g),h \rangle =0$.
\end{proof} \begin{lemma}\label{lem: imageofgj} If $g$ is a special cut vector then
$$P(g) \in \operatorname{span} \{h \colon \operatorname{supp}(g) \subset \operatorname{supp}(h)\}.$$
In particular, $P(g_n)=0$.
\end{lemma}
\begin{proof} If $\operatorname{supp}(h) \subseteq \operatorname{supp}(g)$  or $\operatorname{supp}(h) \cap \operatorname{supp}(g)=\emptyset$,
then, as above,  $\langle P(g),h \rangle =0$, which gives the result.
\end{proof}

The following lemma will be needed in the proof of Theorem~\ref{thm: projnormlowerbound} below.
 \begin{lemma} \label{lem: equivalentnorm}Let $(H_j)_{j=1}^n$ be a chain of cycle vectors  such that $H_j$ is of type $h_j$ and  $\operatorname{supp}(H_j)
 \subset \operatorname{supp}(H_{j+1})$ for each $1 \le j < n$. Then
$$\|\sum_{j=1}^n a_j H_j\|_1 \ge \frac{3}{4} \sum_{j=1}^n |a_j| \|H_j\|_1$$
for all scalars $(a_j)_{j=1}^n$.
\end{lemma} \begin{proof} Note that, for each $2 \le j \le n$,
$$\|H_j|_{\operatorname{supp}( H_{j-1})}\|_1  = \frac{1}{8}\|H_{j}\|_1.$$ Hence \begin{align*}
\|\sum_{j=1}^n a_j H_j\|_1 &= |a_n|\|H_n|_{\operatorname{supp}( H_{n})\setminus\operatorname{supp}( H_{n-1})}\|_1 +\|\sum_{j=1}^{n-1} a_j H_j +a_nH_n|_{\operatorname{supp}( H_{n-1}}\|_1\\
&\ge |a_n|\|H_n\|_1 + \|\sum_{j=1}^{n-1} a_j H_j \|_1 - 2|a_n|\|H_n|_{\operatorname{supp}( H_{n-1})}\|_1\\
&= \frac{3}{4} |a_n|\|H_n\|_1 + \|\sum_{j=1}^{n-1} a_j H_j \|_1.
\end{align*} Iterating this calculation yields the result.
\end{proof}  \begin{theorem} \label{thm: projnormlowerbound} Let $Q$ be any projection from $E_n$  onto $Z_n$. Then $\|Q\|_1 \ge \dfrac38(n+1)$. \end{theorem}
\begin{proof} Let $P$ be the invariant projection associated to $Q$. We shall prove that $\|P\|_1 \ge \dfrac38(n+1)$, which implies the result since
$\|P\|_1\le\|Q\|_1 $.

The analysis of $P$ is very similar to the analysis of $P_n$ in the previous section. In particular, we will  define an auxiliary sequence of vectors $(X_j)_{j=1}^n$ and an auxiliary sequence of scalars $(x_j)_{j=1}^n$. The goal is to \textit{construct} a chain $(S_j)_{j=0}^{n-1}$,  with $S_0= \{e\}$,  such that $\|P(e)\|_1$
 is large, i.e.,   comparable to $\|P_n\|_1$.  This is a  chain which (roughly speaking) maximizes  $\|X_j\|_1$ at each bifurcation.

To that end, we shall inductively  define a chain of cycle vectors $(H_j)_{j=1}^n$   such that $H_j$ is of type $h_j$ and  $\operatorname{supp}(H_j)
 \subset \operatorname{supp}(H_{j+1})$ for each $1 \le j < n$.  To start the induction, set $H_n = h_n$.  To simplify the calculation of the norm we define an equivalent norm $|\!|\!| \cdot |\!|\!|$ on $\operatorname{span}(H_j)_{j=1}^n$ which is easier to work with:
$$|\!|\!| \sum_{j=1}^n a_j H_j |\!|\!| =  \sum_{j=1}^n |a_j| \|H_j\|_1.$$
By Lemma~\ref{lem: equivalentnorm}
$$ \| \sum_{j=1}^n a_j H_j \|_1\le|\!|\!| \sum_{j=1}^n a_j H_j |\!|\!| \le \frac{4}{3} \| \sum_{j=1}^n a_j H_j \|_1.$$
Inductively, we define  vectors $(X_j)_{j=1}^n$  and   a decreasing chain $S_{n-1}\supset S_{n-2} \supset\dots\supset S_1$ such that $S_j$ is the support of a sub-$\mathcal{L}_j$. To  start the inductive definition, set  $$X_n =  \frac{H_n}{\|H_n\|_2^2}\quad\text{and}\quad  x_n =|\!|\!| X_n |\!|\!|= \|X_n\|_1 = (\frac43)^{n-1} $$ and let  $S_{n-1} \subset  \{h_n>0\}$. Set
$$x_{n-1} = |\!|\!| X_n - P(\frac{G_{n-1}}{\|G_{n-1})\|_2^2} )|\!|\!|\vee  |\!|\!|\frac{ X_n }{2} + P(\frac{G_{n-1}}{\|G_{n-1}\|_2^2}) + \frac{H_{n-1}}{\|H_{n-1}\|_2^2}|\!|\!|.$$ Averaging  the two  vectors above and using convexity of $|\!|\!|\cdot |\!|\!| $, \begin{align*}
x_{n-1} &\ge  |\!|\!| \frac34 X_n + \frac12 \frac{H_{n-1}}{\|H_{n-1}\|_2^2}|\!|\!|\\
&= \frac34 |\!|\!| X_n |\!|\!| + \frac12 \| \frac{H_{n-1}}{\|H_{n-1}\|_2^2}\|_1\\ &= \frac34 x_n + \frac12(\frac43)^{n-2}
\end{align*} by \eqref{eq: normestimates}.  If $$x_{n-1}=  |\!|\!|\frac{ X_n }{2} + P(\frac{G_{n-1}}{\|G_{n-1}\|_2^2}) + \frac{H_{n-1}}{\|H_{n-1}\|_2^2}|\!|\!|,$$
set  $$X_{n-1} = \frac{ X_n }{2} + P(\frac{G_{n-1}}{\|G_{n-1}\|_2^2}) + \frac{H_{n-1}}{\|H_{n-1}\|_2^2}$$
and choose $S_{n-2} \subset \{H_{n-1}>0\}$. Otherwise, set
$$X_{n-1}=X_n - P(\frac{G_{n-1}}{\|G_{n-1}\|_2^2})$$ and choose $S_{n-2} \subset S_{n-1}$ disjoint from $\operatorname{supp}(H_{n-1})$.

We now describe the inductive step which is similar to the case $j=n-1$. Suppose  that $1 \le j < n-1$ and that $S_i$,  $X_i$ and $x_i$ have been defined for
$i = j+1,\dots,n$ with $S_j \subset S_{j+1}\subset \dots \subset S_n$ and with $X_i \in \operatorname{span}\{H_k \colon i \le k \le n\}$.
Set $x_i = |\!|\!| X_i |\!|\!|$

 Let $G_j$ and $H_j$  be the cut  and cycle vectors whose support sub-$\mathcal{L}_j$ is  $S_j$.  Note that, by Lemma~\ref{lem: imageofgj},
$$ P(G_j)  \in \operatorname{span} \{H_i \colon  j+1\le i \le n \}$$ and hence
$$x_j =  |\!|\!| X_{j+1} - P(\frac{G_{j}}{\|G_{j}\|_2^2} )|\!|\!|\vee  |\!|\!|\frac{ X_{j+1} }{2} + P(\frac{G_{j}}{\|G_{j}\|_2^2}) + \frac{H_{j}}{\|H_{j}\|_2^2}|\!|\!|$$ is well-defined. Moreover, by convexity, \begin{align*}
x_j &\ge |\!|\!|\frac34 X_{j+1}   +\frac12 \frac{H_{j}}{\|H_{j}\|_2^2}|\!|\!| \\
&= \frac34 |\!|\!|X_{j+1} |\!|\!| + \frac12 \frac{\|H_{j}\|_1}{\|H_{j}\|_2^2}\\
\intertext{(since $X_{j+1} \in  \operatorname{span}\{H_k \colon: j+1\le k \le n \}$)}
&=\frac34 x_{j+1}  +\frac12 (\frac43)^{j-1}.
\end{align*} If $$x_j =|\!|\!|\frac{ X_{j+1} }{2} + P(\frac{G_{j}}{\|G_{j}\|_2^2}) + \frac{H_{j}}{\|H_{j}\|_2^2}|\!|\!|,$$
set $$X_j = \frac{ X_{j+1} }{2} + P(\frac{G_{j}}{\|G_{j}\|_2^2}) + \frac{H_{j}}{\|H_{j}\|_2^2}$$
and choose $S_{j-1} \subset \{H_j > 0\}$. Otherwise, set
$$X_j = X_{j+1} -  P(\frac{G_{j}}{\|G_{j})\|_2^2})$$
and choose $S_{j-1}\subset S_j$ disjoint from $\operatorname{supp}(H_j)$. Note that in both cases we have $X_j \in \operatorname{span}\{H_k \colon j \le k \le n\}$ as required.  This completes the inductive definition. Note that $S_0 = \{e\}$ for some edge vector $e$.
Moreover, using Lemma~\ref{lem: gnhnsigns} we can combine both cases  to obtain, for $1 \le j \le n-1$,
$$X_j = (\frac12)^{\varepsilon_j} X_{j+1} +\operatorname{sgn}(\langle e, G_{j} \rangle)  \frac{P_n(G_{j})}{\|G_{j}\|_2^2} +
 \operatorname{sgn}(\langle e, H_{j} \rangle)  \frac{H_{j}}{\|H_{j}\|_2^2},$$
where $$\varepsilon_j = \begin{cases} 1, &S_{j-1} \subset \operatorname{supp}(H_j),\\
0,   &S_{j-1} \cap \operatorname{supp}(H_j) = \emptyset.\end{cases}$$

Arguing as in the proof of Lemma~\ref{lem: imageofe}  it follows that
\begin{align*}
X_1 &= \sum_{j=1}^{n-1} [\langle e,G_j \rangle \frac{P(G_j)}{\|G_j\|_2^2} + \langle e,H_j \rangle \frac{H_j}{\|H_j\|_2^2}]+ \langle e,H_n \rangle \frac{H_n}{\|H_n\|_2^2}\\
&= P( \sum_{j=1}^{n-1} [\langle e,\frac{G_j}{\|G_j\|_2} \rangle \frac{G_j}{\|G_j\|_2} + \langle e,\frac{H_j}{\|H_j\|_2} \rangle \frac{H_j}{\|H_j\|_2} ]
+ \langle e,\frac{H_n}{\|H_n\|_2} \rangle \frac{H_n}{\|H_n\|_2})\\ &=P(e) \end{align*}
 To see this, note that
 $P(g_n)=0$ by Lemma~\ref{lem: imageofgj} and $P(g)=0$ by Lemma~\ref{lem: nonspecialimage} unless $g$ is a special cut vector of the form $g_j$ for some sub-$\mathcal{L}_j$. Note also that  if $h$ is of the form $h_j$  and $g$  is of the form $g_j$ for  some sub-$\mathcal{L}_j$, then  $\langle e,h \rangle   \ne 0$   only if $h=H_j$ ($1\le j \le n$)  and
$\langle e, g\rangle \ne 0$   only if $g = G_j$ ($1 \le j \le n$). So the above  expression for $X_1$ is simply $P$ applied to the expansion of $e$ with respect to the othogonal basis of $E_n$. 

Finally, \begin{align*} \|P\|_1 & \ge \|P(e)\|_1
 = \|X_1\|_1
\ge \frac34 |\!|\!|X_1 |\!|\!|
= \frac34 x_1
 \ge \frac34 (\frac{n+1}{2}).
\end{align*} The last inequality follows from the solution of the recurrence in  Lemma~\ref{lem: recurrence} since   $$x_j \ge \frac34 x_{j+1} + \frac12 (\frac43)^{j-1}, x_n = (\dfrac43)^{n-1}.$$  \end{proof}

\section{Applications to the transportation cost space of $\mathcal{L}_n$} \label{sec: applications}

\begin{theorem}\label{thm: projconstantLaakso} The projection constant of $\operatorname{Lip}_0(\mathcal{L}_n)$ satisfies
$$ \frac{3n-5}{8} \le \lambda(\operatorname{Lip}_0(\mathcal{L}_n)) \le \frac{n+3}{2}.$$
 \end{theorem}
\begin{proof} Note that $\operatorname{Lip}_0(\mathcal{L}_n)=(\tc(\mathcal{L}_n))^*$ is isometrically isomorphic to $(C_n, \|\cdot\|_\infty) \subset (E_n, \|\cdot\|_\infty)$ by \eqref{E:LFunweigh},  since $C_n = Z_n^\perp$.
Let $P_n$ be the projection from $(E_n, \|\cdot\|_1)$ onto $Z_n$ constructed in Section~\ref{sec: projectionPn}.
Then $I - P_n^*$ is a projection from  $(E_n, \|\cdot\|_\infty)$ onto $Z_n^\perp = C_n$. Thus,
$$\lambda(\operatorname{Lip}_0(\mathcal{L}_n)) \le \|I - P_n^*\| \le 1 + \|P_n\| \le 1 + \frac{n+1}{2} =  \frac{n+3}{2}.$$

Now suppose $Q$ is any projection from  $(E_n, \|\cdot\|_\infty)$ onto $C_n$. Then $I- Q^*$ is a projection from $(E_n, \|\cdot\|_1)$ onto $Z_n$.
So, by Theorem~\ref{thm: projnormlowerbound},
$$\|Q\| \ge \|I - Q^*\| - 1 \ge \frac{3}{8}(n+1)-1 = \frac{3n-5}{8}.$$
So $\lambda(\operatorname{Lip}_0(\mathcal{L}_n)) \ge (3n-5)/8.$
\end{proof}
\begin{corollary}  \label{cor: BMlowerbound} The Banach-Mazur distance from $\tc(\mathcal{L}_n)$ to $\ell_1^N$, where $N(n) = (4\cdot 6^n + 1)/5$ is the  dimension of  $\tc(\mathcal{L}_n)$,
satisfies $$d_{BM}(\tc(\mathcal{L}_n),\ell_1^N) \ge (3n-5)/8.$$
\end{corollary}
\begin{proof}  By duality,
$$d_{BM}(\tc(\mathcal{L}_n),\ell_1^N)=d_{BM}(\operatorname{Lip}_0(\mathcal{L}_n), \ell_\infty^N) \ge \lambda(\operatorname{Lip}_0(\mathcal{L}_n)) \ge  \frac{3n-5}{8}.$$
\end{proof}
\begin{remark} The interpretation of this corollary in terms of  transportation costs is as follows. For  each $1 \le j \le N$,  let $x_j$ be any transportation plan on $\mathcal{L}_n$  of unit cost. Then there exists an absolutely convex combination $\sum_{j=1}^N a_j x_j$ ($\sum_{j=1}^N |a_j| = 1$)  such that
$$\|\sum_{j=1}^N a_j x_j\|_{\tc} \le \frac{8}{3n-5} \qquad(n \ge 2).$$
\end{remark}
In contrast to the diamond graphs $D_n$ \cite[Theorem 6.5]{DKO},  we have not been able to prove a good  upper bound for  the Banach-Mazur distance from $\tc(\mathcal{L}_n)$ to $\ell_1^N$. However,
we  have the following matching upper bound for a linear embedding of $\tc(\mathcal{L}_n)$  into $\ell_1$.
\begin{corollary}  \label{cor: L1embedding} There exists $X_n \subset (E_n,\| \cdot \|_1)$ such that $d_{BM}(\tc(\mathcal{L}_n), X_n) \le
(n+3)/2.$
\end{corollary} \begin{proof} Let $P_n$ be the projection constructed in Section~\ref{sec: projectionPn}. Then, setting $X_n = \operatorname{ker} P_n$,
Theorem~\ref{thm: goodprojection} yields
\[d_{BM}(\tc(\mathcal{L}_n),X_n) =
d_{BM}((E_n/Z_n,\|\cdot\|_{1}),X_n)\le \|I-P_n\|_1 \le
\frac{n+3}{2}.\qedhere\]
\end{proof} \begin{remark}  Actually, as we remarked in the Introduction, for  ever finite metric space $X$,  $\tc(X)$    admits a linear embedding into $L_1[0,1]$ with distortion
$\le C\ln |X|$, see \cite{Cha02,FRT04,IT03}. Corollary~\ref{cor: L1embedding} is  just  a slightly more precise statement of this fact  for $\tc(\mathcal{L}_n)$. \end{remark}

For the diamond graph $D_n$, the transportation cost space $\tc(D_n)$  has a natural  monotone Schauder basis which leads to a matching upper bound for the Banach-Mazur distance. The difficulty in obtaining the same result for $\tc(\mathcal{L}_n)$  stems from the fact that the orthogonal basis of $C_n$ constructed above is \textit{not} a Schauder basis in
the $\tc(\mathcal{L}_n)$ norm.  In fact, the
 collection of special cut vectors  $g_j$ in   $(C_n,\|\cdot\|_{\tc})$ does not  admit a bounded biorthogonal system  (uniformly in $n$).

 To make this precise, for each $1 \le j \le n-1$, let $g_j^i$ ($1 \le i \le 6^{n-j}$)
be an enumeration of the $6^{n-j}$ basis vectors supported on a sub-$\mathcal{L}_j$. Note that $\tc(\mathcal{L}_n)$ is isometrically isomorphic to $(C_n, \|\cdot\|_{\tc})$, where  $\|\cdot\|_{\tc}$ denotes the   quotient norm of
$(E_n, \|\cdot\|_1)/Z_n$.

\begin{proposition} Suppose  $g_n^* \in (C_n, \|\cdot\|_\infty)$  satisfies
$$g_n^*(g_n) = \|g_n\|_{\tc} \quad\text{and}\quad g_n^*(g_j^i)=0\qquad (1 \le j \le n-1, 1 \le i \le 6^{n-j}).$$
Then $\|g_n^*\|_\infty \ge (4/3)^{n-1}.$
\end{proposition}
\begin{proof}  Note that
$\|g_n\|_{\tc} = \|g_n\|_1.$ This follows easily  from convexity since each $h \in Z_n$ has a symmetric distribution relative  to $g_n$ (see Figure~\ref{fig: gnhn}) and so $\|g_n + h\|_1 \ge \|g_n\|_1$. (In fact, one can show that $\|g^i_j\|_{\tc} = \|g^i_j\|_1$ for all $i,j$ but this is not needed for the proof.)
Note also that  (see Figures~\ref{fig: fn} and \ref{fig: gnhn})
$$\|f_n - \frac{1}{2}g_n\|_1 = \frac{3}{4}\|f_n\|_1.$$  Applying this to each sub-$\mathcal{L}_{n-1}$ of $\mathcal{L}_n$  (see Figure~\ref{fig: gnhn}) gives
 $$\|g_n - \frac{1}{2} \sum_i \varepsilon^i_{n-1} g^i_{n-1} \|_1 = \frac{3}{4}\|g_n\|_1$$
for some choice of signs $\varepsilon^i_{n-1} = \pm1$. Repeating this argument,  we get
$$\|g_n - \frac{1}{2}[ \sum_i \varepsilon^i_{n-1} g^i_{n-1}   + \frac{3}{2} \sum_i \varepsilon^i_{n-2} g^i_{n-2}]\|_1 = (\frac{3}{4})^2\|g_n\|_1$$
for some choice of $\varepsilon^i_j \in \{-1,0,1\}$. In general, we get for each $1 \le k \le n-1$,
$$\|g_n - \frac{1}{2}[\sum_{j=k}^{n-1} (\frac{3}{2})^{n-1-j} (\sum_i \varepsilon^i_{j} g^i_{j})] \|_1 = (\frac{3}{4})^{n-k}\|g_n\|_1$$
for some choice of   $\varepsilon^i_j \in \{-1,0,1\}$. Hence \begin{equation} \label{eq: badequiv}
\|g_n - \frac{1}{2}[\sum_{j=1}^{n-1} (\frac{3}{2})^{n-1-j} (\sum_i \varepsilon^i_{j} g^i_{j})] \|_{\tc} \le (\frac{3}{4})^{n-1}\|g_n\|_1
=  (\frac{3}{4})^{n-1}\|g_n\|_{\tc}. \end{equation}
The desired result follows.
  \end{proof} \begin{remark} The proof shows that the
 collection of special cut vectors  $g_j$ does not admit a bounded biorthogonal system (uniformly in $n$)  for its span in  $(C_n,\|\cdot\|_1)$ .
  In particular, the orthogonal  basis of $C_n$ constructed above is not a Schauder basis (uniformly in $n$) in $(C_n, \|\cdot\|_1)$.

 Moreover, \eqref{eq: badequiv} show that the equivalence constant of the basis of $\|\cdot\|_1$-normalized (or $\|\cdot\|_{\tc}$-normalized)  special cut vectors with the unit vector basis of $\ell_1$ is at least $(4/3)^{n-1}$.

On the other hand, the orthogonal basis of $Z_n$  constructed above  is a   monotone Schauder  basis for $(Z_n, \|\cdot\|_1)$.
This allows an estimate from above for $d_{BM}(Z_n, \ell_1^N)$.
\begin{proposition} $d((Z_n,\|\cdot\|_1), \ell_1^N ) \le 2n$, where $N = \operatorname{dim}(Z_n) =(6^{n}-1)/5$.
\end{proposition} \begin{proof}
  For $1 \le j \le n$, let
$H_j = (h^i_j)_{i=1}^{6^{n-j}}$ be an enumeration of the $Z_n$ basis vectors of the form $h_j$ for some sub-$\mathcal{L}_j$.  Since each $h^i_j$  is symmetric on its support
sub-$\mathcal{L}_j$ it follows by convexity that $\cup_{j=0}^{n-1} H_{n-j}$ is a monotone basis of $(Z_n\|\cdot\|)$. Moreover, $\{h^i_j/\|h^i_j\|_1 \colon 1 \le i \le 6^{n-j}\}$
 is $1$-equivalent to the unit vector basis of $\ell_1^{6^j}$ since these vectors have disjoint supports.
  Let $x \in Z_n$ and write
$ x = \sum_{k=0}^{n-1} x_k$, where   $x_k \in \operatorname{span}(H_{n-k})$. Then, by monotonicity of the basis,
$$ \sum_{k=0}^{n-1} \|x_k\| \ge \|x\| \ge \frac{1}{2} \max_{0 \le k \le n-1} \|x_k\| \ge \frac{1}{2n} \sum_{k=0}^{n-1} \|x_k\|$$
Hence  $\cup_{j=0}^{n-1} H_{n-j}$ is  $2n$-equivalent to a suitably  scaled standard basis of $\ell_1^n$, which gives the result.
\end{proof} \end{remark} \section{Multi-branching diamond graphs}\label{sec: diamondgraphs}
In this section we sharpen some of the results of \cite[Section~6]{DKO}.
\begin{theorem} \label{thm: diamondgraphs} For each  $k\ge2$  and $n\ge1$,
$$\lambda(\operatorname{Lip}_0(D_{n,k})) = \frac{2k-2}{2k-1}n + \frac{4k^2-6k+3}{(2k-1)^2} + \frac{2k-2}{(2k-1)^2} \frac{1}{(2k)^n}.$$
In particular, for $k=2$ and $n \ge 1$,
$$\lambda(\operatorname{Lip}_0(D_{n})) =\frac{2n}{3} + \frac{7}{9} + \frac{2}{9} 4^{-n}.$$
\end{theorem}
 \begin{proof} Let us recall the representation of $D_{n,k}$ used in \cite{DKO}.  We identify the edge space of $D_{n,k}$ with a subspace of $L_1[0,1]$ as follows.  For $n=1$ and $1 \le j \le k$
we
identify the pair of  edge vectors of the $j^{th}$ path of length $2$ from the `top' to the `bottom' vertex with the $L_1$-normalized indicator functions
$2k 1_{(j-1)/k, (2j-1)/(2k)]}$ and $2k 1_{((2j-1)/(2k), j/k]}$. For $n \ge2$,  the edge space of $D_{n,k}$ is obtained from that of
$D_{n,k-1}$ by subdividing the intervals corresponding to edge vectors of $D_{n,k-1}$ into $2k$ subintervals each of length
$(2k)^{-n}$.  Each of the $k$ consecutive disjoint  pairs of $L_1$-normalized indicator functions of the subintervals corresponds to each
pair of  edge vectors of the $k$ paths of length $2$ from
the top and bottom vertices of the copy of $D_{1,k}$ which  replaces the edge vector  of $D_{n-1,k}$   corresponding to the interval of length $(2k)^{n-1}$ which is  subdivided. We have now identified
 the  edge vectors of $D_{n,k}$ with the  $L_1$-normalized indicator functions
$$ e_{n,j} = (2k)^n 1_{((j-1)/(2k)^n, j/(2k)^n]} \quad (1 \le j \le (2k)^n).$$

A  basis for the cycle space corresponds to the $L_\infty$-normalized system $\cup_{i=1}^n \{ g_{i,j}  \colon 1 \le j \le
(2k )^{i-1}(k-1)\}$, where, setting $j = a(k-1)+b$ with $0 \le a \le (2k)^{i-1}-1$ and $1 \le b \le k-1$,
$$g_{i,j} =(2k)^{-i}( e_{i,a2^k+2b-1} + e_{i,a 2^k + 2b} - e_{i,a 2^k + 2b + 1}- e_{i,a2^k+2b+2}).$$
For  $k \ge 3$, note that $g_{i,j}$ \textit{ overlaps}  with $g_{i,j+1}$  when $b \le k-2$, and hence this is not an orthogonal basis.

An orthogonal basis for the  cut space   corresponds to the $L_\infty$-normalized system $\{h_0\} \cup \cup_{i=1}^n \{ h_{i,j} \colon 1 \le j \le (2k)^i\}$, where
$h_0 = 1_{[0,1]}$, and
$$ h_{i,j}
= (2k)^{-i}(e_{i,2j-1}-e_{i,2j}).$$

 Let $G$ be the group of automorphisms of the edge space  generated by
those automorphisms   which interchange (by translations)  the intervals $\{g_{i,j}>0\}$ and $\{g_{i,j}<0\}$   or the sets $\{ h_{i,j}>0\}$ and $\{h_{i,j}<0\}$. Then (as observed in \cite{DKO})
arguing as in Lemma~\ref{lem: nonspecialimage},
the orthogonal projection $P_{n,k}$ onto the cut space is the \textit{unique} $G$-invariant  projection onto the cut space.  First, let us compute the
$\|\cdot\|_1$- norm of $P_{n,k}$.
Note that  $$P_{n,k}(e_{n,1}) =h_0 + \frac{1}{2} \sum_{i=1}^n (2k)^i h_{i,1}.$$
An elementary calculation which we omit yields
$$\|P_{n,k}(e_{n,1})\|_1 = \frac{2k-2}{2k-1}n + \frac{4k^2-6k+3}{(2k-1)^2}+ \frac{2k-2}{(2k-1)^2} \frac{1}{(2k)^n}.$$
Now suppose $1 \le j \le (2n)^k$. For $1 \le i \le n$,   let $\operatorname{supp}(e_{n,j}) \subset \operatorname{supp}(h_{i,r(i)})$. Then
$$P_{n,k}(e_{n,j})= h_0 + \frac{1}{2} \sum_{i=1}^n  \operatorname{sgn}(\langle e_{n,j}, h_{i,r(i)}\rangle) (2k)^i h_{i,r(i)}.$$
So  $P_{n,k}(e_{n,j})$ has the same \textit{distribution }as  $P_{n,k}(e_{n,1})$.
In particular,
$\|P_{n,k}(e_{n,j})\|_1 = \|P_{n,k}(e_{n,1})\|_1$.  Hence
$$\|P_{n,k}\|_1 = \max_{1 \le j \le (2n)^k} \|P_{n,k}(e_{n,j})\|_1 = \|P_{n,k}(e_{n,1})\|_1.$$
Finally,  since  $P_{n,k}$ is the unique $G$-invariant projection onto the cut space and is self-adjoint,
$$\lambda(\operatorname{Lip}(D_{n,k}))=\|P_{n,k}\|_\infty = \|P_{n,k}\|_1= \frac{2k-2}{2k-1}n + \frac{4k^2-6k+3}{(2k-1)^2} + \frac{2k-2}{(2k-1)^2} \frac{1}{(2k)^n}.$$
\end{proof}
As a corollary, we get an improvement on \cite[Theorem 6.10]{DKO}.
\begin{corollary} For each $n \ge1$ and $k \ge 2$, the Banach-Mazur distance $d_{n,k}$ from the transportation cost  space $\tc(D_{n,k})$ to the $\ell_1^N$ space of the same dimension satisfies
$$ d_{n,k} \ge \frac{2k-2}{2k-1}n + \frac{4k^2-6k+3}{(2k-1)^2} + \frac{2k-2}{(2k-1)^2} \frac{1}{(2k)^n}.$$
\end{corollary}

\section{Characterization of finite trees in terms of their transportation cost spaces}\label{S:Tree}

The following result is well known.

\begin{proposition}\label{P:TCTree}  Let $M$ be a finite metric space with $n$ elements. The space $\tc(M)$ is isometric to $\ell_1^{n-1}$
if and only if $M$ is a weighted tree (the weight of an edge is
the distance between its ends) with its shortest path distance.
\end{proposition}

Apparently for finite metric spaces it is folklore.
The earliest proof of the ``if'' part we are aware of is
\cite[Corollary 3.6]{God10}. Its more general version for
infinite metric spaces was proved in \cite{DKP16}. Our goal is to
give a direct proof of the ``only if'' part. A simple direct proof
of the ``if'' part can be found in \cite[Proposition 2.1]{DKO}.

\begin{proof} We suppose that $\tc(M)$ is isometric to $\ell_1^{n-1}$ and prove
that this implies that $T$ is isometric to a weighted tree.

We may and shall identify the metric space $M$ with a complete
weighted graph, whose vertex set is $M$ and for which the weight
of an edge is the distance between its ends. In such a case the
metric of $M$ coincides with the weighted graph distance of this
graph.

An edge $uv$ in this weighted graph is called {\it essential} if
and only if $d(u,v)<d(u,w)+d(w,v)$ for every $w\in
M\backslash\{u,v\}$, or, equivalently, if the weighted graph
distance of this graph will change if the edge $uv$ is deleted.

It is well known (and easy to check) that for a finite metric
space a vector $f$ is an extreme point of the unit ball of
$\tc(M)$ if and only if $f=(\1_u-\1_v)/d(u,v)$ for some essential
edge $uv$ in the described weighted graph (this result is known in
a more general form \cite{AP20}, in which it is far from being
easy).

Since $\ell_1^{n-1}$ has $(n-1)$ symmetric pairs of extreme
points, we conclude that the weighted graph corresponding to $M$
has $(n-1)$ essential edges. Since it is clear that the set of
essential edges has to connect the graph, we get that the set of
essential edges in $M$ forms a spanning tree. Recalling the
definition of essential edges, we derive that the metric of $M$ is
the distance of the weighted tree formed by essential
edges.\end{proof}

\begin{corollary} The space $\tc(M)$ with $|M|=n$ has between
$(n-1)$ and $\frac{n(n-1)}2$ symmetric pairs of extreme points and
thus is a quotient of $\ell_1^d$ for $(n-1)\le
d\le\frac{n(n-1)}2$.
\end{corollary}

\begin{proof} In fact, the number of essential edges in a weighted connected simple graph
with $n$ vertices can be any number between $(n-1)$ and
$\frac{n(n-1)}2$. This follows from the following easy
observations: (a) All edges in an unweighted (equivalently, a
weighted graph with all weights equal to $1$) connected simple
graph are essential, and the number of such edges can be any
number between $(n-1)$ and $\frac{n(n-1)}2$. (b) Essential edges
induce a connected spanning graph, and thus there should be at
least $(n-1)$ of them.
\end{proof}

\section{Isometric copies of $\ell_\infty^3$ and $\ell_\infty^4$
in $\tc(M)$ on finite metric spaces}\label{S:Linfty}

One of the results of \cite{KMO} is a construction of finite
metric spaces for which $\tc(M)$ contains isometric copies of
$\ell_\infty^3$ and $\ell_\infty^4$. The goal of this last section
is to provide a simpler constructions of such spaces. We show that

\begin{enumerate}

\item There exists a $6$-point set $T$ such that $\tc(T)$ contains
$\ell_\infty^3$ isometrically.

\item There exists an $8$-point set $F$ such that $\tc(F)$
contains $\ell_\infty^4$ isometrically.

\end{enumerate}

Below we describe the metric spaces and the transportation
problems spanning $\ell_\infty^3$ and $\ell_\infty^4$,
respectively. We leave it as an exercise the straightforward
verification of the equality
\[\left\|\sum_{i=1}^k\theta_if_i\right\|=1\]
for $k=3$ or $k=4$, and $\theta_i=\pm1$.

The description of the metric space $T$:

\begin{table}[h]
\center{
\begin{tabular}{c|c|c|c|c|c|c}
 & a & b & c & d & e & f \\ \hline
a & 0 & 1 & 1 & 1 & 1/2 & 1/2\\ \hline b & 1 & 0 & 1 & 1 & 1/2 & 1/2\\
\hline
c & 1 & 1 & 0 & 1 & 1/2 & 1/2\\
\hline
d & 1 & 1 & 1 & 0 & 1/2 & 1/2\\
\hline
e & 1/2 & 1/2 & 1/2 & 1/2 & 0 & 1\\
\hline
f & 1/2 & 1/2 & 1/2 & 1/2 & 1 & 0\\
\hline

\end{tabular}
\vskip0.3cm \caption[b]{Distances}\label{T:add} \vskip0.5cm}
\end{table}

The description of three transportation problems on $T$ spanning
$\ell_\infty^3$:

\begin{table}[h]
\center{
\begin{tabular}{c|c|c|c|c|c|c}
 & a & b & c & d & e & f \\ \hline
$f_1$ & 1/2 & -1/2 & 1/2 & -1/2 & 0 & 0\\ \hline
$f_2$ & 1/2 & 1/2 & -1/2 & -1/2 & 0 & 0\\
\hline
$f_3$ & 0 & 0 & 0 & 0 & 1 & -1\\
\hline

\end{tabular}
\vskip0.3cm \caption[b]{Values of transportation
problems}\label{T:addF3} \vskip0.5cm}
\end{table}

\newpage

The description of the metric space $F$:

\begin{table}[h]
\center{
\begin{tabular}{c|c|c|c|c|c|c|c|c}
 & a & b & c & d & e & f & g & h\\ \hline
a & 0 & 1 & 1 & 1 & 1/2 & 1/2 & 1/2 & 1/2\\ \hline b & 1 & 0 & 1 & 1 & 1/2 & 1/2 & 1/2 & 1/2\\
\hline
c & 1 & 1 & 0 & 1 & 1/2 & 1/2 & 1/2 & 1/2\\
\hline
d & 1 & 1 & 1 & 0 & 1/2 & 1/2 & 1/2 & 1/2\\
\hline
e & 1/2 & 1/2 & 1/2 & 1/2 & 0 & 1 & 1 & 1\\
\hline
f & 1/2 & 1/2 & 1/2 & 1/2 & 1 & 0 & 1 & 1\\
\hline
g & 1/2 & 1/2 & 1/2 & 1/2 & 1 & 1 & 0 & 1\\
\hline
h & 1/2 & 1/2 & 1/2 & 1/2 & 1 & 1 & 1 & 0\\
\hline

\end{tabular}
\vskip0.3cm \caption[b]{Distances}\label{T:add4} \vskip0.5cm}
\end{table}

The description of four transportation problems on $F$ spanning
$\ell_\infty^4$:

\begin{table}[h]
\center{
\begin{tabular}{c|c|c|c|c|c|c|c|c}
 & a & b & c & d & e & f & g & h\\ \hline
$f_1$ & 1/2 & -1/2 & 1/2 & -1/2 & 0 & 0 & 0 & 0\\ \hline
$f_2$ & 1/2 & 1/2 & -1/2 & -1/2 & 0 & 0 & 0 & 0\\
\hline
$f_3$ & 0 & 0 & 0 & 0 & 1/2 & -1/2 & 1/2 & -1/2\\
\hline
$f_4$ & 0 & 0 & 0 & 0 & 1/2 & 1/2 & -1/2 & -1/2\\
\hline

\end{tabular}
\vskip0.3cm \caption[b]{Values of transportation
problems}\label{T:addF4} \vskip0.5cm}
\end{table}

\thanks{\textbf{Acknowledgements:}
The authors thank the referee for a very careful reading of  the manuscript  and for making  numerous corrections and helpful suggestions which resulted in a much clearer presentation.
The second author acknowledges the support from the Simons
Foundation under Collaborative Grant No 636954. The third author
gratefully acknowledges the support by the National Science
Foundation grant NSF DMS-1953773.}

\end{document}